\documentclass{amsart}

\usepackage{import}
\usepackage{xcolor}

\usepackage[utf8]{inputenc}
\usepackage{amsthm,amsmath,amsfonts,amssymb}
\usepackage{amsrefs,hyperref,cleveref}
\usepackage{multirow}
\usepackage{graphicx}
\usepackage{mathtools}
\usepackage{enumerate}
\usepackage[myheadings]{fullpage}

\newtheorem{theorem}{Theorem}[section]
\newtheorem{assumption}{Assumption}
\newtheorem{lemma}[theorem]{Lemma}
\newtheorem{proposition}[theorem]{Proposition}
\newtheorem{corollary}[theorem]{Corollary}

\theoremstyle{definition}
\newtheorem{definition}[theorem]{Definition}

\newtheorem{example}{Example}[section]
\newtheorem{remark}{Remark}

\newcommand{\F}{\mathbb{F}}
\newcommand{\R}{\mathbb{R}}
\newcommand{\C}{\mathbb{C}}

\newcommand{\E}{\mathbb{E}}

\newcommand{\U}{\mathbb{U}}
\newcommand{\Y}{\mathbb{Y}}

\renewcommand{\P}{\mathbb{P}}
\newcommand{\Z}{\mathbb{Z}}

\newcommand{\prob}{\mathrm{Prob}}

\renewcommand{\vec}[1]{\boldsymbol{#1}}
\newcommand{\cov}{\mathrm{Cov}}

\newcommand{\e}{\varepsilon}

\newcommand{\1}{\mathbf{1}}
\newcommand{\bA}{\mathbf{A}}

\newcommand{\bi}{\mathbf{i}}

\newcommand{\bl}{\boldsymbol{\ell}}
\newcommand{\bm}{\mathbf{m}}

\newcommand{\br}{\mathbf{r}}
\newcommand{\bs}{\mathbf{s}}
\newcommand{\bu}{\mathbf{u}}
\newcommand{\bv}{\mathbf{v}}
\newcommand{\bx}{\mathbf{x}}
\newcommand{\by}{\mathbf{y}}
\newcommand{\bz}{\mathbf{z}}

\newcommand{\fA}{\mathfrak{A}}
\newcommand{\fB}{\mathfrak{B}}
\newcommand{\fC}{\mathfrak{C}}

\newcommand{\fI}{\mathfrak{I}}

\newcommand{\fM}{\mathfrak{M}}

\newcommand{\fP}{\mathfrak{P}}
\newcommand{\fR}{\mathfrak{R}}
\newcommand{\fU}{\mathfrak{U}}
\newcommand{\fV}{\mathfrak{V}}

\newcommand{\fY}{\mathfrak{Y}}
\newcommand{\fZ}{\mathfrak{Z}}

\newcommand{\fp}{\mathfrak{p}}

\newcommand{\fx}{\mathfrak{x}}
\newcommand{\fy}{\mathfrak{y}}

\newcommand{\cA}{\mathcal{A}}
\newcommand{\cB}{\mathcal{B}}
\newcommand{\cC}{\mathcal{C}}

\newcommand{\cE}{\mathcal{E}}
\newcommand{\cF}{\mathcal{F}}

\newcommand{\cI}{\mathcal{I}}

\newcommand{\cK}{\mathcal{K}}
\newcommand{\cL}{\mathcal{L}}

\newcommand{\cP}{\mathcal{P}}

\newcommand{\cR}{\mathcal{R}}

\newcommand{\cT}{\mathcal{T}}
\newcommand{\cU}{\mathcal{U}}

\newcommand{\wh}[1]{\widehat{#1}}
\newcommand{\wt}[1]{\widetilde{#1}}
\renewcommand{\vec}[1]{\boldsymbol{#1}}
\newcommand{\Ai}{\mathrm{Ai}}

\numberwithin{equation}{section}

\title{Fluctuations of $\beta$-Jacobi Product Processes}
\author{Andrew Ahn}

\begin{document}
\maketitle

\begin{abstract}
We study Markov chains formed by squared singular values of products of truncated orthogonal, unitary, symplectic matrices (corresponding to the Dyson index $\beta = 1,2,4$ respectively) where time corresponds to the number of terms in the product. More generally, we consider the $\beta$-Jacobi product process obtained by extrapolating to arbitrary $\beta > 0$. When the time scaling is preserved, we show that the global fluctuations are jointly Gaussian with explicit covariances. For time growing linearly with matrix size, we show convergence of moments after suitable rescaling. When $\beta = 2$, our results imply that the right edge converges to a process which interpolates between the Airy point process and a deterministic configuration. This process connects a time-parametrized family of point processes appearing in the works of Akemann-Burda-Kieburg and Liu-Wang-Wang across time. In the arbitrary $\beta > 0$ case, our results show tightness of the particles near the right edge. The limiting moment formulas correspond to expressions for the Laplace transform of a conjectural $\beta$-generalization of the interpolating process.
\end{abstract}

\section{Introduction}

Let $\F_\beta$ be the real, complex, or quaternion skew field and $\U_\beta(L)$ be the $L$-dimensional orthogonal, unitary, or symplectic group for $\beta = 1,2,4$ respectively. Consider the random process $(\by^{(T)})_{T \in \Z_+}$ formed by the squared singular values $\by^{(T)} := (y_1^{(T)} \ge \cdots \ge y_N^{(T)})$ of matrices
\begin{align} \label{eq:s_value}
Y_T := X_T \cdots X_1
\end{align}
where $X_1,X_2,\ldots$ are independent random matrices such that $X_T$ is $N_T\times N_{T-1}$ and time $T$ corresponds to the number of factors. If the distributions of $X_1,X_2,\ldots$ are taken to be invariant under the right action of $\U_\beta$, then $(\by^{(T)})_{T \in \Z_+}$ is a discrete time Markov chain. Our key example of a right $\U_\beta$ invariant random matrix is a \emph{truncated $\U_\beta$ matrix} which is an $N' \times N$ submatrix $X$ of a Haar distributed $\U_\beta(L)$ matrix. If we write $M = L-N'$, $\alpha = N'-N+1$, the squared singular values of $X$ are of the form $(1,\ldots,1,x_1,\ldots,x_{\min(M,N)})$ and the joint density is proportional to
\begin{align} \label{eq:bJac_intro}
\prod_{1 \le i < j \le \min(M,N)} |x_i - x_j|^\beta \prod_{i=1}^{\min(M,N)} x_i^{(\beta/2) \alpha - 1} (1 - x_i)^{(\beta/2)(|M-N| + 1) - 1} \, dx_i
\end{align}
for $(x_1,\ldots,x_{\min(M,N)}) \in [0,1]^{\min(M,N)}$.

Extrapolating the density \eqref{eq:bJac_intro} to arbitrary $\beta > 0$, we obtain the \emph{$\beta$-Jacobi ensemble} --- a random $N$-particle system which can be viewed as a Coulomb gas model where $\beta$ is the inverse temperature (see e.g. \cite{Ox}*{Chapter 20}). For $\beta = 1,2,4$, we can consider the \emph{$\beta$-Jacobi product process} $(\by^{(T)})_{T \in \Z_+}$ from taking $X_1,X_2,\ldots$ to be truncated $\U_\beta$ matrices. The central objects of this article are Markov processes $(\by^{(T)})_{T \in \Z_+}$ obtained by extending the $\beta$-Jacobi product process to arbitrary $\beta > 0$. In particular, we study the fluctuations of $(\by^{(T)})_{T \in \Z_+}$ under several limit regimes. Our extrapolation relies on a $\beta > 0$ generalization of matrix products under right $\U_\beta$-symmetry, introduced by Gorin-Marcus \cite{GM}. Before expanding on our asymptotic results, we describe this extrapolation in greater detail.

Although the notion of $\U_\beta$ exists only for $\beta = 1,2,4$, we can extend the notion of right $\U_\beta$ invariant matrix products to arbitrary $\beta > 0$ by viewing it as a random operation on vectors as follows. Fix $\beta \in \{1,2,4\}$ and suppose that $X$ is a fixed $N\times N$ matrix with $\F_\beta$ entries and squared singular values $\bx$, and $Y$ is an $N\times N$ truncated $\U_\beta$ matrix with squared singular values $\by$. Let $\bx \boxtimes_\beta \by$ denote the squared singular values of $YX$. Set $\Y_N$ to be the set of partitions of length at most $N$ and $1^N$ to denote $N$-vector of all ones. It is known that
\begin{align} \label{eq:Jack_relation}
\frac{J_\kappa(\bx;\beta/2)}{J_\kappa(1^N;\beta/2)} \cdot \E \left[ \frac{J_\kappa(\by;\beta/2)}{J_\kappa(1^N;\beta/2)} \right] = \E \left[ \frac{J_\kappa(\bx \boxtimes_\beta \by;\beta/2)}{J_\kappa(1^N;\beta/2)} \right], \quad \quad \kappa \in \Y_N
\end{align}
where $J_\kappa(\cdot;\beta/2)$ is the Jack function with parameter $\beta/2$ \cite{Mac}*{Chapter VII}; this is a consequence of the fact that these normalized Jack functions are zonal spherical functions for the Gelfand pair $(\mathrm{GL}_N(\F_\beta), \U_\beta(N))$ for $\beta = 1,2,4$. Moreover, \eqref{eq:Jack_relation} uniquely identifies the probability distribution of $\bx \boxtimes_\beta \by$. Since the Jack functions are defined for arbitrary $\beta > 0$, \eqref{eq:Jack_relation} suggests a means to extend $\bx \boxtimes_\beta \by$ to $\beta > 0$. Our first result which we present confirms that this can be done.

\begin{theorem} \label{thm:existence}
Let $\beta > 0$. Suppose $\bx \in [0,1]^N$ is fixed and $\by$ is distributed as the $\beta$-Jacobi ensemble for some parameters $\alpha,M$ as in \eqref{eq:bJac_intro}. This uniquely determines the distribution of a random $N$-vector $\bx \boxtimes_\beta \by$ such that \eqref{eq:Jack_relation} holds.
\end{theorem}

This gives us a Markov transition kernel from $\bx$ to $\bx \boxtimes_\beta \by$ via the distribution of $\by$. The assumption that $\bx \in [0,1]^N$ can be lifted to $\R_{\ge 0}^N$ without much difficulty, though this is not required for our purposes. We note that this theorem is quite nontrivial. For generally distributed $\by$, it is only a conjecture that the measure associated to $\bx \boxtimes_\beta \by$ via \eqref{eq:Jack_relation} is positive, see \cite{GM} and references therein for further details. In our setting where $\by$ is a $\beta$-Jacobi ensemble, we obtain positivity through a connection with Macdonald symmetric functions. The general case where $X$ and $Y$ are rectangular matrices can be reduced to the case where $X$ and $Y$ are square via matrix projections, see \Cref{sec:matrix}.

From this extrapolation, we obtain Markov chains $(\by^{(T)})_{T \in \Z_+}$ for each $\beta > 0$. More specifically, if $\bx^{(1)},\bx^{(2)},\ldots$ are independent random $N$-vectors where $\bx^{(T)}$ is distributed as the $\beta$-Jacobi ensemble with parameters $\alpha_T,M_T$, then the \emph{$\beta$-Jacobi product process} with parameters $(\alpha_T,M_T)_{T \in \Z_+}$ is the Markov chain $(\by^{(T)})_{T \in \Z_+}$ where $\by^{(T)} = \bx^{(T)} \boxtimes_\beta \by^{(T-1)}$ and $\by^{(1)} = \bx^{(1)}$. Thus we obtain a $\beta > 0$ deformation of the squared singular values of products of truncated $\U_\beta$-matrices.

The main objective of this article is to study the fluctuations of $(\by^{(T)})_{T\in\Z_+}$ in the limit as $N\to\infty$. In particular, we consider the following two settings: (i) global fluctuations where we do not rescale time as $N$ grows and (ii) local fluctuations at the right edge (that is, of the rightmost particles) where time grows linearly with $N$. Along the way, we prove a limit shape result for arbitrary $\beta > 0$ which extends known results \cites{Vo91,CS06} for $\beta = 1,2,4$.

We show that the global fluctuations are described by a Gaussian process whose covariance function exhibits logarithmic correlation on short-scales. This is reminiscent of the Gaussian free field --- a distinguished $2$-dimensional, conformally invariant, log-correlated Gaussian field which appears in the fluctuations of many $2$-d models from statistical mechanics.

For local fluctuations, our results are twofold. We consider the regime where time $T$ tends to infinity linearly with $N$. First, we demonstrate that for $\beta = 2$, the fluctuations of the right edge are described by a process in time $\wh{T}$ whose fixed-time marginals interpolate between the Airy point process as $\wh{T} \to 0$ and a deterministic, ``picket fence'' configuration as $\wh{T} \to \infty$. Second, for arbitrary $\beta > 0$, we show tightness of the point process at the right edge where we expect to see a $\beta$-generalization of the interpolating process. Moreover, we provide exponential moment formulas for this conjectural limit process.

From the method's perspective, our goal is to combine ideas about Macdonald processes (special processes derived from the two parameter family of Macdonald symmetric functions), $\beta$-Jacobi ensembles, and products of matrices into a unified picture \cites{BC,BCGS,BG1,GM,BGS}. Our approach uses Macdonald symmetric, Jack symmetric, Heckman-Opdam hypergeometric functions, and their correspondence \cite{GM} to product processes in order to obtain moment formulas. By the asymptotic analysis of these moment formulas, we access the fluctuations of $\beta$-Jacobi product processes.

As far as the author is aware, these results are the first to describe fluctuations for products of $\beta$-ensembles beyond $\beta = 2$. For local fluctuations at $\beta = 2$, previous works \cites{ABK,LWW} established \emph{fixed-time convergence results} to a family of point processes indexed by $\wh{T}$ for the case of Ginibre matrices. These point processes interpolate between the Airy point process and picket fence statistics as $\wh{T}$ ranges from $0$ to $\infty$. Our main theorems on local fluctuations extend these results to \emph{joint convergence across time} $\wh{T}$ for Ginibre \emph{and} Jacobi matrices. As a consequence, the interpolating process that appears in our work links together this $\wh{T}$-parametrized family of point processes across time; we provide more details below. The appearance of this process is independent of the choice of Jacobi parameters. This suggests that there may be a wider universality class of products of matrices with generic distributions where this interpolating process appears. For general $\beta > 0$, we find universality of limiting moment formulas which do not depend on the parameters of the $\beta$-Jacobi product process. We conjecture that there exists a $\beta$-generalization of the interpolating process. Under this conjecture, these moment formulas are expressions for the Laplace transforms of the generalized interpolating processes evaluated at positive integer values.

We now proceed to a more detailed discussion of our main results and methods. Our main results on global fluctuations are provided in \Cref{sec:intro:global} along with additional background. Similarly, \Cref{sec:intro:local} contains some background and our main results on local fluctuations at the right edge. We conclude the introduction with a description of our methods in \Cref{sec:intro:method}.

\subsection{Global Fluctuations} \label{sec:intro:global}

Our first asymptotic regime preserves the time scaling as the number of particles $N$ tends to $\infty$. We begin with several known limit shape and fluctuation results under this regime.

In a series of breakthrough articles (see e.g. \cites{Vo87,Vo91}), Voiculescu established that the squared singular values of products of certain large unitarily invariant random matrices concentrate around a limit shape. These results have since been generalized  to include orthogonally and symplectically invariant matrix ensembles including the $\beta$-Hermite, Laguerre, and Jacobi ensembles for $\beta = 1,4$, see e.g. \cite{CS06}*{Theorem 5.1}.

Under general assumptions on unitarily invariant random matrices $X_1,X_2,\ldots$, the global fluctuations of the squared singular values of $Y_T$ are known to be Gaussian due to Collins-Mingo-Sniady-Speicher \cite{CMSS07}*{Theorems 7.9 and 8.3} via second-order freeness, and due to Guionnet-Novak \cite{GN} via Schwinger-Dyson equations. An explicit form for the covariance was recently discovered by Gorin-Sun \cite{GSu} who used a difference operators approach on multivariate Bessel functions. In particular, they found that the covariance can be identified by a Gaussian field related to the Gaussian free field, discussed further below.

On a related note, a variety of authors established the Gaussianity of global fluctuations 
for eigenvalues of other types of product matrix ensembles. By asymptotic analysis of the Stieltjes transform, Vasilchuk \cite{Va16} proved a central limit theorem for linear statistics of eigenvalues of a certain product of unitary matrices. A recent, general result of Coston-O'Rourke \cite{CO19} showed Gaussian fluctuations for the eigenvalues of products of Wigner matrices.

While the fluctuations for $\beta = 2$ are well-studied, the same cannot be said even for $\beta = 1,4$. It appears that the exact form of global fluctuations for $\beta = 1,4$ were not accessed prior to this work; see however \cites{MP13,Red15} for generalizations of the notion of second order freeness for orthogonal and symplectic ensembles. This gap in the literature is one of the motivations for this work.

From the perspective of $\beta$-ensembles, our work seeks to extend results on fluctuations of $\beta$-ensembles to $\beta$-product processes. Several prior works established Gaussianity of fluctuations for ordinary $\beta$-ensembles. The first work of this kind was due to Johansson \cite{Jo98} where Gaussian fluctuations for $\beta$-ensembles with analytic potentials are shown using loop equations. This approach was further extended in the articles of Kriecherbauer-Shcherbina \cite{KS}, Borot-Guionnet \cites{BoGu1,BoGu2}, Shcherbina \cite{Shch}, Borodin-Gorin-Guionnet \cite{BGG}, and Dimitrov-Knizel \cite{DK}. Dumitriu-Paquette \cite{DP} established global fluctuations for the general $\beta$-Jacobi ensemble through a tridiagonal representation of this ensemble \cite{DE}. Borodin-Gorin \cite{BG1} generalized this result for the $\beta$-Jacobi corners process, showing Gaussian free field fluctuations using an approach related to ours involving Macdonald processes (discussed further in \Cref{sec:intro:method}).

We first provide a limit shape theorem for arbitrary $\beta > 0$, extending known results for $\beta = 1,2,4$. Let
\begin{align}
\mathrm{m}_{\by^{(T)}} := \frac{1}{N} \sum_{i=1}^N \delta_{y_i^{(T)}}.
\end{align}

\begin{theorem}[Limit Shape] \label{thm:lln}
Suppose $(\by^{(T)})_{T \in \Z_+}$ is distributed as the $N$-particle $\beta$-Jacobi product process with parameters $(\alpha_T:= \alpha_T(N), M_T := M_T(N))_{T \in \Z_+}$ and $\wh{\alpha}_T, \wh{M}_T \ge 0$ such that
\[ \lim_{N\to\infty} (\alpha_T/N,M_T/N) = (\wh{\alpha}_T,\wh{M}_T) \]
for each $T \in \Z_+$. Then for any positive integers $k,T$ there exists a probability measure $\bm_{(\wh{\alpha}_\tau,\wh{M}_\tau)_{\tau=1}^T}$ (independent of $\beta$) such that
\[ \lim_{N\to\infty} \mathrm{m}_{\by^{(T)}} = \bm_{(\wh{\alpha}_\tau,\wh{M}_\tau)_{\tau=1}^T} \]
weakly in probability. Moreover, we have
\[ \int \! x^k \, d\bm_{(\wh{\alpha}_\tau,\wh{M}_\tau)_{\tau=1}^T} = - \frac{1}{k} \cdot \frac{1}{2\pi\bi} \oint \left( \frac{v}{v+1} \prod_{\tau = 1}^T \frac{v - \wh{\alpha}_\tau }{v - \wh{\alpha}_\tau - \wh{M}_\tau} \right)^k \, dv \]
where the contour is positively oriented around the pole at $-1$ but does not enclose $\wh{\alpha}_\tau + \wh{M}_\tau$ for $1 \le \tau \le T$.
\end{theorem}

\begin{remark}
Consequently, for $T = 1$, we have some limit $\bm_{\wh{\alpha}_1,\wh{M}_1}$. Then the limit of products of $\beta$-ensembles can be expressed as
\[ \bm_{(\wh{\alpha}_\tau, \wh{M}_\tau)_{\tau=1}^T} = \bm_{\wh{\alpha}_1,\wh{M}_1} \boxtimes \cdots \boxtimes \bm_{\wh{\alpha}_T,\wh{M}_T} \]
where $\boxtimes$ denotes the \emph{free multiplicative convolution} (see e.g. \cite{Vo87}).
\end{remark}

Our main result for global asymptotics states that the fluctuations are Gaussian with explicit covariances.

\begin{theorem}[Global Fluctuations] \label{thm:clt}
Suppose $(\by^{(T)})_{T \in \Z_+}$ is distributed as the $N$-particle $\beta$-Jacobi product process with parameters $(\alpha_T:= \alpha_T(N), M_T := M_T(N))_{T \in \Z_+}$ and $\wh{\alpha}_T,\wh{M}_T \ge 0$ such that
\[ \lim_{N\to\infty} (\alpha_T/N,M_T/N) = (\wh{\alpha}_T,\wh{M}_T) \]
for each $T \in \Z_+$. Then for any positive integers $m$, $k_1,\ldots k_m$ and $T_1 \ge \ldots \ge T_m$, we have that the random vector
\begin{align} \label{eq:clt_vector}
\left( \int \! x^{k_1} \, d\mathrm{m}_{\by^{(T_1)}}(x) - \E \int \! x^{k_1} \, d\mathrm{m}_{\by^{(T_1)}}(x), \ldots, \int \! x^{k_m} \, d\mathrm{m}_{\by^{(T_m)}}(x) - \E \int \! x^{k_m} \, d\mathrm{m}_{\by^{(T_m)}}(x) \right)
\end{align}
converges in distribution to a Gaussian vector as $N\to\infty$ such that the covariance between the $i$th and $j$th component is given by
\begin{align} \label{eq:asymp_cov}
\frac{(\beta/2)^{-1}}{(2\pi\bi)^2} \oint \oint  \frac{1}{(v_2 - v_1)^2} \left( \frac{v_1}{v_1 + 1} \prod_{\tau=1}^{T_i} \frac{v_1 - \wh{\alpha}_\tau}{v_1 - \wh{\alpha}_\tau - \wh{M}_\tau} \right)^{k_i} \left( \frac{v_2}{v_2 + 1} \prod_{\tau=1}^{T_j} \frac{v_2 - \wh{\alpha}_\tau}{v_2 - \wh{\alpha}_\tau - \wh{M}_\tau} \right)^{k_j} \, dv_1 \, dv_2
\end{align}
where the $v_2$-contour encloses the $v_1$-contour, both the $v_1,v_2$-contours are positively oriented around $-1$, but the $v_1$-contour does not contain $\wh{\alpha}_\tau + \wh{M}_\tau$ for $1 \le \tau \le T_i$ and the $v_2$-contour does not contain $\wh{\alpha}_\tau + \wh{M}_\tau$ for $1 \le \tau \le T_j$.
\end{theorem}

We observe that the covariance depends on the symmetry class $\beta$ only through a factor of $\beta^{-1}$. This is a common feature feature among $\beta$-ensembles in the literature; compare with e.g. \cite{BG1}. In the case $\beta = 2$, our result intersects that of \cite{GSu}. In particular, this means that the covariance can be described in terms of a Gaussian process whose covariance function is logarithmic on short-scales, as in \cite{BG1}*{Proof of Theorem 4.13}, \cite{GSu}*{Proof of Corollary 4.10}. This is related to the appearance of the Gaussian free field in $\beta$-ensembles which has a distinguished logarithmic covariance structure given by the Green's kernel of the Laplacian on the upper half plane with Dirichlet boundary conditions.

\subsection{Local Fluctuations for Growing Products} \label{sec:intro:local}

Our second asymptotic regime takes the number of products $T$ and particles $N$ to grow linearly with respect to one another. Under this regime, we study the fluctuations of the right edge.

A recent result due to Akemann-Burda-Kieburg \cite{ABK} and Liu-Wang-Wang \cite{LWW}*{Theorem 1.2} considers this asymptotic regime for squared singular values of $Y_T$ as defined by \eqref{eq:s_value} with $X_1,X_2,\ldots$ taken to be $N\times N$ complex Ginibre matrices. We recall that an $N\times N$ complex Ginibre matrix is a matrix of i.i.d. standard complex Gaussians. Define $(\xi_1 := \xi_1(N) \ge \cdots \ge \xi_N := \xi_N(N))$ from the squared singular values $\by^{(T)} = (y_1^{(T)},\ldots,y_N^{(T)})$ of $Y_T$ via
\begin{align} \label{eq:original_rescale}
y_i^{(T)} = N^{T+1} e^{\xi_i}.
\end{align}
When $\lim_{N\to\infty} T/N = \wh{T} > 0$, these authors showed that in the limit $N\to\infty$, the point process $\xi_1,\xi_2,\ldots$ converges to a $\wh{T}$-parametrized determinantal point process \cite{AGZ}*{C4.2} $\fx_i^{(\wh{T})},\fx_2^{(\wh{T})},\ldots$ whose correlation kernel is given by
\begin{align} \label{eq:corr_kernel}
K_{\wh{T}}(\fx,\fy) = \int_{1-\bi\infty}^{1+\bi\infty} \frac{ds}{2\pi\bi} \oint \frac{dt}{2\pi\bi} \frac{1}{s-t} \frac{\Gamma(t)}{\Gamma(s)} \frac{e^{\frac{\wh{T} (s^2 - s)}{2} - \fy s}}{e^{\frac{\wh{T} (t^2 - t)}{2} - \fx s}}
\end{align}
where the $t$-contour is to the left of $1$, starting from $-\infty - \bi \e$, positively looping around $0,-1,-2,\ldots$, and then going to $-\infty + \bi \e$. We note that the description above differs from that of Liu-Wang-Wang \cite{LWW}*{Equation 1.8} by a translation of $\wh{T}/2$. Akemann-Burda-Kieburg have a different form \cite{ABK}*{Equation 19} for the correlation kernel of $\fx_1^{(\wh{T})}, \fx_2^{(\wh{T})},\ldots$ where their point process differs from that of ours by a multiplicative factor of $\wh{T}^{2/3}$ and a translation by $1 + \log \wh{T}$.

The following convergence statements are given in \cite{LWW}*{Theorem 3.2} with a proof sketch. As $\wh{T} \to 0$, the process $\zeta_1^{(\wh{T})},\zeta_2^{(\wh{T})},\ldots$ defined by
\[ \fx_i^{(\wh{T})} = 2^{-1/3} \wh{T}^{2/3} \zeta_i^{(\wh{T})} + 1 + \log \wh{T} \]
converges to the Airy point process \cite{AGZ}*{p232-234} defined by the correlation kernel
\[ K_{\mathrm{Airy}}(\zeta,\eta) = \frac{\Ai(\zeta) \Ai'(\eta) - \Ai(\eta) \Ai'(\zeta)}{\zeta - \eta}. \]
As $\wh{T} \to +\infty$, we have the convergence
\[ \left\{ \wh{T}^{-1} \fx_i^{(\wh{T})} \right\}_{i=1}^\infty \to \left\{ -i + \frac{1}{2} \right\}_{i=1}^\infty \]
to deterministic particles, referred to as \emph{picket fence statistics} due to the associated measure being a semi-infinite set of equally spaced unit delta masses. Furthermore, as $\wh{T} \to +\infty$, the fluctuations of the largest particle are Gaussian. More precisely,
\[ \wh{T}^{-1/2} \left( \fx_1^{(\wh{T})} + \frac{\wh{T}}{2} \right) \to \mbox{standard Gaussian}. \]
While the convergence of the second largest, third largest, etc eigenvalues are not explicitly shown in \cite{ABK}, \cite{LWW}, there is good evidence (e.g. \cite{LWW}*{Theorem 1.1}) that each should be Gaussian and a rigorous proof should be accessible through the correlation kernels. We note that Liu-Wang-Wang also consider Ginibre matrices of different rectangular sizes and show that the point process $(\fx^{(\wh{T})}_i)_{i=1}^\infty$ appears universally in this setting \cite{LWW}*{Section 3.4}.

Although the correlation kernel \eqref{eq:corr_kernel} determines $\fx^{(\wh{T})} := (\fx_i^{(\wh{T})})_{i=1}^\infty$, we use an alternative characterization via the Laplace transform. By combining our result (in particular \Cref{thm:local_cpx}) for the Ginibre case with that of \cite{LWW}, we obtain a formula for the Laplace transform of $\fx^{(\wh{T})}$.

\begin{theorem}[Formula for Laplace Transform] \label{thm:laplace}
For $c_1,\ldots,c_m > 0$, we have
\begin{align*}
\begin{multlined}
\E\left[ \prod_{i=1}^m \sum_{j=1}^\infty e^{c_i \fx_j^{(\wh{T})}} \right] \\
= \frac{1}{(2\pi\bi)^m} \oint \cdots \oint \det \left[ \frac{1}{u_i - u_j - c_j} \right]_{i,j=1}^m \prod_{i=1}^m \frac{\Gamma(-u_i-c_i)}{\Gamma(-u_i)} \exp\left[ -\wh{T} \left( \frac{c_i(c_i + 1)}{2} + c_i u_i \right) \right] du_i
\end{multlined}
\end{align*}
where the $u_i$-contour $\fU_i$ is a positively oriented contour around $-c_i,-c_i+1,-c_i+2,\ldots$ which starts and ends at $+\infty$, and $\fU_i$ is enclosed by $\fU_j - c_i$, $\fU_j + c_j$ for $j > i$.
\end{theorem}

Taking $\wh{T} \to 0$ and $c_i$ to grow linearly with $\wh{T}^{-2/3}$, we can recover a formula for the Laplace transform of the Airy point process (see \cite{BG2}*{Equation 14}). Taking $\wh{T} \to +\infty$ and $c_i$ to decay linearly with $\wh{T}^{-1/2}$, we obtain the Laplace transform of a Gaussian; if instead we let $c_i$ decay as $\wh{T}^{-1}$, we obtain the Laplace transform for picket fence statistics. In the former two cases, these limits are directly accessible from the contour integral formula. In the latter case, one must take the residue expansion.

Our main results for $\beta = 2$ consider a natural extension of $\fx^{(\wh{T})} = (\fx^{(\wh{T})}_1,\fx^{(\wh{T})}_2,\ldots)$ across time $\wh{T}$.

\begin{definition}
By an \emph{interpolating process} (that is, interpolating between the Airy point process and picket fence statistics), we mean a process $(\fx^{(\wh{T})})_{\wh{T} > 0} := (\fx_i^{(\wh{T})})_{\wh{T} > 0, i \in \Z_+}$ defined by the joint Laplace transform
\begin{align*}
\begin{multlined}
\E\left[ \prod_{i=1}^m \sum_{j=1}^\infty e^{c_i \fx_j^{(\wh{T}_j)}} \right] \\
= \frac{1}{(2\pi\bi)^m} \oint \cdots \oint \det \left[ \frac{1}{u_i - u_j - c_j} \right]_{i,j=1}^m \prod_{i=1}^m \frac{\Gamma(-u_i-c_i)}{\Gamma(-u_i)} \exp \left[ - \wh{T}_i \left( \frac{c_i(c_i + 1)}{2} + c_i u_i \right) \right] du_i
\end{multlined}
\end{align*}
where $\wh{T}_1 \ge \cdots \ge \wh{T}_m > 0$, $c_1,\ldots,c_m > 0$, the $u_i$-contour $\fU_i$ is a positively oriented contour around $-c_i,-c_i+1,-c_i+2,\ldots$ which starts and ends at $+\infty$, and $\fU_i$ is enclosed by $\fU_j - c_i$, $\fU_j + c_j$ for $j > i$.
\end{definition}

Indeed, \Cref{thm:laplace} implies that the distribution of $(\fx^{(\wh{T})})_{\wh{T} > 0}$ for a fixed time slice $\wh{T}$ is exactly the point process with correlation kernel \eqref{eq:corr_kernel}. Given a vector $U = (u_1,\ldots,u_k)$, we write $\log U := (\log u_1,\ldots,\log u_k)$. We now state the main results for local fluctuations for $\beta = 2$.

\begin{theorem}[$\beta = 2$ Ginibre, Edge Fluctuations] \label{thm:local_cpx_intro}
Suppose $(\by^{(T)})_{T\in\Z_+}$ are the squared singular values of $Y_T$ defined as in \eqref{eq:s_value} where $X_1,X_2,\ldots$ are independent $N\times N$ complex Ginibre matrices. Then
\[ \lim_{N\to\infty} \log\left( \frac{1}{N^{\lfloor N \wh{T} \rfloor + 1}} \by^{(\lfloor N\wh{T} \rfloor)}\right) = \fx^{(\wh{T})} \]
in finite dimensional distributions across time $\wh{T} > 0$.
\end{theorem}

\begin{figure}[h]
    \centering
    \includegraphics[scale=0.38]{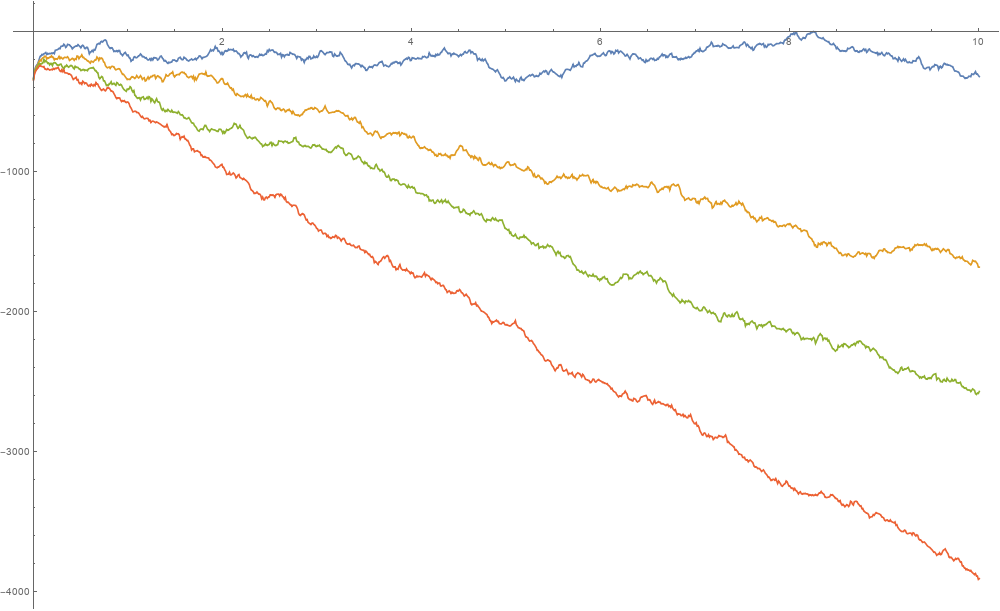}
    \caption{$\log\left( \frac{1}{N^{\lfloor N \wh{T}\rfloor + 1}} y_i^{(\lfloor N \wh{T} \rfloor)} \right)$, $i = 1,2,3,4$, $\wh{T} \in (0,10]$ from a $100 \times 100$ Ginibre matrix}
\end{figure}

\begin{theorem}[$\beta = 2$ Jacobi, Edge Fluctuations] \label{thm:local_cpx_jacobi_intro}
Suppose $(\by^{(T)})_{T \in \Z_+}$ are the squared singular values of $Y_T$ defined as in \eqref{eq:s_value} where $X_1,X_2,\ldots$ are independent $N\times N$ matrices with squared singular values distributed as the ($\beta = 2$) Jacobi ensemble with parameters $\alpha_T := \alpha_T(N), M_T := M_T(N)$ such that
\[ \liminf_{N\to\infty} \left( \inf_{T > 0} M_T/N \right) > 0, \quad \liminf_{N\to\infty} \left( \inf_{T > 0} \alpha_T \right) > 0 \]
and there exists $\gamma:\R_+ \to \R_+$ such that for each $\wh{T} > 0$
\[ \gamma(\wh{T}) = \lim_{N\to\infty} \sum_{\tau = 1}^{\lfloor \wh{T} N \rfloor} \left( \frac{1}{N + \alpha_\tau - 1} - \frac{1}{N + M_\tau + \alpha_\tau - 1} \right) \]
where we assume the right hand side limit exists. Then
\[ \lim_{N\to\infty} \log\left( \left( \frac{1}{N} \prod_{\tau=1}^{\lfloor N \wh{T} \rfloor} \frac{N+M_\tau+\alpha_\tau-1}{N+\alpha_\tau-1} \right) \by^{(\lfloor N \wh{T} \rfloor)} \right) = \fx^{(\gamma(\wh{T}))} \]
in finite dimensional distributions across time $\wh{T} > 0$.
\end{theorem}

\Cref{thm:local_cpx_intro,thm:local_cpx_jacobi_intro} generalize \cite{LWW}*{Theorem 1.2} in two directions. In one direction, we show convergence of the right edge jointly across time instead of a single fixed time. To the best of the author's knowledge, this is the first appearance of the limit process $(\fx^{(\wh{T})})_{\wh{T} > 0}$. In another direction, we establish that this convergence is universal among products of Jacobi ensembles. Our methods also extend to Laguerre ensembles, although we do not state any results in this direction beyond the Ginibre case (note that Ginibre singular values are a special case of the Laguerre ensemble).

We expect that $(\fx^{(\wh{T})})_{\wh{T} > 0}$ is determinantal. We believe that the correlation kernel for $(\fx^{(\wh{T})})_{\wh{T} > 0}$ can be computed via the correlation kernels arising in \cite{BGS} for the truncated unitary product process. However, working with Laplace transforms has the advantage of generalizing to arbitrary $\beta > 0$.

We now move on to the main result for arbitrary $\beta > 0$.

\begin{theorem}[$\beta > 0$, Edge Fluctuations] \label{thm:local_intro}
Let $k_1,\ldots,k_m$ be positive integers. Suppose $(\by^{(T)})_{T \in \Z_+}$ is distributed as the $N$-particle $\beta$-Jacobi ensemble with parameters $\alpha_T := \alpha_T(N), M_T := M_T(N)$ such that
\[ \liminf_{N\to\infty} \left( \inf_{T > 0} M_T/N \right) > 0, \quad \liminf_{N\to\infty} \left( \inf_{T > 0} \alpha_T \right) > 0 \]
and there exists $\gamma:\R_+ \to \R_+$ such that for each $\wh{T} > 0$
\[ \gamma(\wh{T}) = \lim_{N\to\infty} \sum_{\tau = 1}^{\lfloor \wh{T} N \rfloor} \left( \frac{1}{N + \alpha_\tau - 1} - \frac{1}{N + M_\tau + \alpha_\tau - 1} \right) \]
where we assume the right hand side limit exists. Then
\begin{align*}
& \lim_{N\to\infty} \E\left[ \prod_{i=1}^m \left( \frac{1}{N} \prod_{\tau=1}^{\lfloor N \wh{T}_i \rfloor} \frac{N+M_\tau+\alpha_\tau-1}{N+\alpha_\tau-1} \right)^{k_i} \int \! x^{k_i} d\mathrm{m}_{\by^{(\lfloor N \wh{T}_i \rfloor)}}(x) \right] \\
&\quad \quad = \frac{(\beta/2)^{-m}}{(2\pi\bi)^{\sum_{i=1}^m k_i}} \oint \cdots \oint \prod_{\substack{1 \le i,i' \le m, 1 \le j \le k_i, 1 \le j' \le k_{i'} \\ (i,j) < (i,j')}} \frac{(u_{i',j'} - u_{i,j})(u_{i',j'} - u_{i,j} + 1 - \beta/2)}{(u_{i',j'} - u_{i,j} + 1)(u_{i',j'} - u_{i,j} - \beta/2)} \\
& \quad \quad \quad \quad \times \prod_{i=1}^m \left( \prod_{a=1}^{k_i-1} \frac{1}{u_{i,a} - u_{i,a+1} - 1 + \beta/2} \right) \left( \prod_{j=1}^{k_i} \frac{e^{-\gamma(\wh{T}_i)(\beta/2)^{-1} u_{i,j}}}{(\beta/2)^{-1}u_{i,j}} \, du_{i,j} \right)
\end{align*}
where the $u_{i,j}$-contour is positively oriented around $0$, and the $u_{i',j'}$-contour encloses a $\max(\beta/2,1)$-neighborhood of the $u_{i,j}$-contour for $(i,j) < (i',j')$ in lexicographical order.
\end{theorem}

We expect the right hand side of \Cref{thm:local_intro} describes the Laplace transforms of finite dimensional distributions of an interpolating process which is a $\beta$-deformation of the $\beta = 2$ case. However, it is not clear that the right hand side determines a limiting process since the moment problem is indeterminate. Despite the indeterminacy, our result implies tightness.

\begin{corollary}
For $(\by^{(T)})_{T \in \Z_+}$ distributed as in \Cref{thm:local_intro}, the finite dimensional distributions of
\begin{align} \label{eq:e_jacobi}
\log\left( \left( \frac{1}{N} \prod_{\tau=1}^{\lfloor N \wh{T} \rfloor} \frac{N+M_\tau+\alpha_\tau-1}{N+\alpha_\tau-1} \right) \by^{(\lfloor N \wh{T} \rfloor)} \right)
\end{align}
are tight in $N$.
\end{corollary}

We conjecture that \eqref{eq:e_jacobi} converges to a universal limit process $(\fx^{(\wh{T},\beta)})_{\wh{T} > 0}$ with the property that $\wh{T} \to 0$ yields the $\beta$-Airy point process and $\wh{T} \to +\infty$ yields a Gaussian limit under proper rescaling. Assuming this conjecture, the right hand side of \Cref{thm:local_intro} gives formulas for Laplace transform of the conjectural point process
\[ \E \left[ \prod_{i=1}^m \sum_{j=1}^\infty e^{k_i \fx_i^{(\gamma(\wh{T}),\beta)}} \right] \]
evaluated at positive integer values $(k_1,\ldots,k_m) \in (\Z_+)^m$. In addition, this conjecture implies that if $k = k_1 = \cdots = k_m$ are taken to grow on the order $\wh{T}^{-2/3}$ as $\wh{T} \to 0$, then we should obtain the Laplace transform of the $\beta$-Airy point process (see \cite{GSh}) in the limit.

\subsection{Method} \label{sec:intro:method}

Our method uses the formalism of Macdonald processes which are stochastic processes with special algebraic properties inherited from the Macdonald symmetric functions. In particular, certain observables of Macdonald processes are accessible through operators which diagonalize the Macdonald symmetric functions. Finding suitable representations for the action of these operators, such as contour integral formulas, provides access to asymptotics of these processes. This idea was pioneered by \cite{BC} and further extended in \cites{BG1,GZ,Ahn1}. In this work, we establish that the $\beta$-Jacobi product process can be realized as a limit of Macdonald processes. Through the Macdonald process formalism, we obtain moment formulas for the $\beta$-Jacobi product process. We are thus able to access global limits for constant order $T$ and local limits for $T$ of order $N$.

The connection between Macdonald processes and $\beta$-ensembles was first observed by Borodin-Gorin \cite{BG1}. They showed that the Heckman-Opdam limit (see \Cref{sec:background}) of certain Macdonald processes produce the $\beta$-Jacobi corners process --- a multilevel extension of the $\beta$-Jacobi ensemble. Using formulas for observables of Macdonald processes, they accessed the global fluctuations of the $\beta$-Jacobi corners process. On the other hand, Borodin-Gorin-Strahov \cite{BGS} recently established that squared singular values of products of truncated unitary matrices ($\beta = 2$ in our model) are a limit of Schur processes (a special case of Macdonald processes). We unify these two connections, finding that the $\beta$-Jacobi product process is a limit of Macdonald processes, thus obtaining observables for the product process. We note that the observables we use are derived from a different set of operators than that of \cite{BG1}. The operators we use were introduced by Negut \cite{N} in an algebraic setting, reexpressed as contour integral formulas by Gorin-Zhang \cite{GZ} for the asymptotic analysis of the $\beta$-Jacobi corners process, and further applied to a broad class of Macdonald processes by the author \cite{Ahn1}. These operators have the advantage of giving exact moment formulas of Macdonald processes and their degenerations, thus are useful tools in streamlining asymptotic analyses.

The $\beta$-Jacobi product process is constructed by generalizing the notion of matrix products with right $\U_\beta$ symmetry. Expanding on our discussion earlier, this generalization was introduced by Gorin-Marcus in \cite{GM} where they considered measures $\sigma_{\bx,\by}^\beta$ with total mass $1$, determined by
\[ \frac{J_\kappa(\bx;\beta/2)}{J_\kappa(1^N;\beta/2)} \frac{J_\kappa(\by;\beta/2)}{J_\kappa(1^N;\beta/2)} = \int \frac{J_\kappa(\bu;\beta/2)}{J_\kappa(1^N;\beta/2)} \, d\sigma_{\bx,\by}^\beta(\bu), \quad \quad \kappa \in \Y_N \]
for fixed $\bx,\by \in \R_{\ge 0}^N$. For $\beta = 1,2,4$, $\sigma_{\bx,\by}^\beta$ coincides with the distribution of squared singular values of $\mathrm{diag}(\by) U \mathrm{diag}(\bx)$ where $U$ is Haar $\U_\beta$ distributed. However, $\sigma_{\bx,\by}^\beta$ are only known to be signed measures for general $\beta > 0$. Through the connection between $\beta$-Jacobi ensembles and Macdonald processes, we establish that if we integrate $\sigma_{\bx,\by}^\beta$ in $\by$ with respect to a $\beta$-Jacobi ensemble, the resulting measure is a probability measure. This is precisely our result \Cref{thm:existence} which states that if $\by$ is distributed as a $\beta$-Jacobi ensemble, the random operation $\bx \boxtimes_\beta \by$ is well-defined. Our proof relies on describing the Markov transition kernels of the $\beta$-Jacobi product process in terms of expectations of Jack symmetric functions and demonstrating convergence of analogous expectations on the side of Macdonald processes. These prelimit Markov transition kernels preserve the Macdonald process structure. This is the key point which allows us to realize $\beta$-Jacobi product processes as limits of Macdonald processes.

With this limit relation established, we obtain moment formulas for the $\beta$-Jacobi product process through formulas for moments of Macdonald processes. Global and local fluctuations are obtained by the moment's method. The global asymptotic analysis uses techniques related to the approaches of \cites{BG1,GZ,Ahn1}. Our method of accessing fluctuations of the right edge (as $T,N \to\infty$) relies on our observation that as $T$ grows, moments of a fixed order are dominated by the largest particles. As a comparison, to access fine asymptotics of the largest particles for Wigner matrices via moments, the moments must be taken of order $N^{2/3}$ where $N$ is the size of the matrix; in fact, the universality of the Airy point process at the spectral edge of Wigner matrices was first shown by Sinai-Soshnikov \cite{SS} and Soshnikov \cite{Sosh} using high order moments. For further background on the moment's method, see the survey \cite{Sod} and references therein.

Currently, our methods are only applicable for regimes where the moments are of constant order, with the exception of $\beta = 2$. For example, we cannot access local edge fluctuations for finite $T$ as this would require the moments to grow on the order $N^{2/3}$. The main obstruction for a growing order of moments is that our moment formulas are given by contour integrals whose dimensions match the order of the moments. At present, we are unable to control the asymptotics of these contour integrals if the dimension is increasing. When $\beta = 2$, the formulas simplify and dimensions reduce. This is the reason why our local tightness results for arbitrary $\beta > 0$ are upgraded to full convergence results for $\beta = 2$. In fact, for $\beta = 2$, the dimensions of the moment formulas are stable enough that any $T,N$ regime should be analyzable via our approach.

The remainder of this article is organized as follows. In \Cref{sec:background}, we provide the necessary background on Macdonald symmetric, Jack symmetric and Heckman-Opdam hypergeometric functions along with references for additional study. We detail the connections between these functions and the $\beta$-Jacobi product process in \Cref{sec:process}, and obtain moment formulas in \Cref{sec:moments} by passing through the Macdonald process formalism. In the remaining \Cref{sec:global,sec:local}, we prove our main theorems on global and local fluctuations respectively via asymptotic analysis of the moment formulas.

\section*{Acknowledgments}
I would like to thank my advisor Vadim Gorin for suggesting the study of matrix products in orthogonal and symplectic symmetry classes, for useful discussions, and giving feedback on several drafts. The author was partially supported by National Science Foundation Grant DMS-1664619.

\section{Preliminaries on Special Functions} \label{sec:background}

\subsection{Symmetric Functions}
Let $\Y$ denote the set of partitions. We represent $\lambda \in \Y$ by the nondecreasing sequence $(\lambda_1,\lambda_2,\ldots)$ of its parts. Let $\ell(\lambda) := \#\{i \ge 1: \lambda_i \ne 0\}$ and $|\lambda| := \sum_{i \ge 1} \lambda_i$. Denote by $\Lambda$ the algebra over $\C$ of symmetric functions in countably many variables $x_1,x_2,\ldots$. Define $p_0 := 1$ and
\[ p_k := \sum_{i \ge 1} x_i^k, ~~~k \in \Z_+, \quad \quad \quad \quad p_\lambda := \prod_{i = 1}^{\ell(\lambda)} p_{\lambda_i}, ~~~\lambda \in \Y. \]
Then $\{p_\lambda\}_{\lambda \in \Y}$ forms a linear basis of $\Lambda$. Fixing $0 < q,t < 1$, we have the scalar product
\[ \langle p_\lambda, p_\mu \rangle_{(q,t)} = \delta_{\lambda\mu} \prod_{i=1}^{\ell(\lambda)} \frac{1 - q^{\lambda_i}}{1 - t^{\lambda_i}} \prod_{i=1}^\infty i^{m_i(\lambda)} m_i(\lambda)! \]
where $m_i(\lambda)$ is the multiplicity of $i$ in $\lambda$. We drop the subscript $(q,t)$ when the dependence is clear.

The \textit{Macdonald symmetric functions} $\{P_\lambda(X;q,t)\}_{\lambda \in \Y}$ are the uniquely defined family of homogeneous symmetric functions satisfying
\[ \langle P_\lambda(X;q,t), P_\mu(X;q,t) \rangle = 0 \]
for $\lambda \neq \mu$ such that the leading monomial of $P_\lambda$ is $x_1^{\lambda_1} x_2^{\lambda_2} \cdots$ with respect to lexicographical ordering of the powers $(\lambda_1,\lambda_2,\ldots)$. This implies that $\{P_\lambda(X;q,t)\}_{\lambda \in \Y}$ forms a linear basis for $\Lambda$ with dual basis $\{ Q_\lambda(X;q,t) \}_{\lambda \in \Y}$ where $Q_\lambda(X;q,t)$ is a multiple of $P_\lambda(X;q,t)$. Given $\mu,\nu \in \Y$, $P_\mu(X;q,t)P_\nu(X;q,t)$ expands in the basis $\{P_\lambda(X;q,t)\}_{\lambda \in \Y}$ such that the coefficient of $P_\lambda(X;q,t)$ is
\[ \langle P_\mu(X;q,t) P_\nu(X;q,t), Q_\lambda(X;q,t) \rangle. \]
This coefficient is nonzero only if
\[ |\mu| + |\nu| = |\lambda|, \quad \quad \mu,\nu \subset \lambda \]
where $\mu \subset \lambda$ means
\[ \mu_i \le \lambda_i, \quad i \ge 1, \]
see \cite{Mac}*{Chapter VI (7.4)} for details.

The \emph{skew Macdonald symmetric functions} $P_{\lambda/\mu}(X;q,t), Q_{\lambda/\mu}(X;q,t)$ are defined by
\begin{align*}
\begin{split}
\langle P_{\lambda/\mu}(X;q,t), Q_\nu(X;q,t) \rangle & = \langle P_\lambda(X;q,t), Q_\mu(X;q,t) Q_\nu(X;q,t) \rangle, \quad \nu \in \Y, \\
\langle Q_{\lambda/\mu}(X;q,t), P_\nu(X;q,t) \rangle & = \langle Q_\lambda(X;q,t), P_\mu(X;q,t) P_\nu(X;q,t) \rangle, \quad \nu \in \Y.
\end{split}
\end{align*}
Then $P_{\lambda/\mu}(X;q,t) \ne 0$ only if $\lambda \supset \mu$, and likewise for $Q_{\lambda/\mu}(X;q,t)$.

Let $\Y_N$ denote the set of partitions of length $\le N$, $\Lambda_N$ the algebra over $\C$ of symmetric polynomials in $N$-variables $x_1,\ldots,x_N$ and $\pi_N: \Lambda \to \Lambda_N$ the restriction homomorphism which effectively takes $x_{N+1} = x_{N+2} = \cdots = 0$. Then
\[ P_\lambda(x_1,\ldots,x_N;q,t) := \pi_N P_\lambda(X;q,t) \]
for $\lambda \in \Y_N$ form a basis for $\Lambda_N$. We have
\[ a_1,\ldots,a_N \ge 0 ~~~\implies~~~ P_{\lambda/\mu}(a_1,\ldots,a_N;q,t) \ge 0; \]
see \cite{Mac}*{Chapter VI,(7.9') \& (7.14')} for the ingredients to prove this nonnegativity. For formal variables $X = (x_1,x_2,\ldots)$ and $Y = (y_1,y_2,\ldots)$, we have the identity
\begin{align} \label{eq:cauchy1}
\sum_{\lambda \in \Y} P_\lambda(X;q,t) Q_\lambda(Y;q,t) = \prod_{i,j=1}^\infty \frac{(tx_iy_j;q)_\infty}{(x_iy_j;q)_\infty} =: \Pi(X,Y;q,t)
\end{align}
where $(a;q)_\infty := \prod_{i \ge 0} (1 - aq^i)$, see \cite{Mac}*{Chapter VI,(4.13)}. Replacing $X$ with $(a_1,\ldots,a_M,0,\ldots)$ and $Y$ with $(b_1,\ldots,b_N,0,\ldots)$ such that $\sup_{i,j} |a_ib_j| < 1$, the sum on the left converges and the identity holds analytically. From \cite{Mac}*{Chapter VI, Section 7.4}, we have
\begin{align} \label{eq:branch}
P_{\lambda/\nu}(X,Y;q,t) & = \sum_{\mu \in \Y} P_{\lambda/\mu}(X;q,t) P_{\mu/\nu}(Y;q,t), \\ \label{eq:flipid}
\sum_{\lambda \in \Y} P_\lambda(X) Q_{\lambda/\mu}(Y) &= \Pi(X,Y;q,t) P_\mu(X).
\end{align}

Define the normalized Macdonald symmetric polynomials
\[ \wh{P}_\lambda(x_1,\ldots,x_N;q,t) = \frac{P_\lambda(x_1,\ldots,x_N;q,t)}{P_\lambda(1,t,\ldots,t^{N-1};q,t)}, ~~~\lambda \in \Y_N \]
which satisfy the label-variable symmetry \cite{Mac}*{Chapter VI, (6.6)} for $\lambda,\mu \in \Y_N$
\begin{align} \label{eq:varind1}
\wh{P}_\lambda(q^{\mu_1} t^{N-1},q^{\mu_2} t^{N-2},\ldots,q^{\mu_N};q,t) = \wh{P}_\mu(q^{\lambda_1} t^{N-1},q^{\lambda_2} t^{N-2},\ldots,q^{\lambda_N};q,t).
\end{align}
Since the normalized Macdonald symmetric functions form a  linear basis for the algebra of symmetric functions, there exist $c^\lambda_{\mu\nu}(\wh{P};q,t)$ satisfying
\[ \wh{P}_\mu(x_1,\ldots,x_N;q,t) \wh{P}_\nu (x_1,\ldots,x_N;q,t) = \sum_\lambda c^\lambda_{\mu\nu}(\wh{P};q,t) \wh{P}_\lambda (x_1,\ldots,x_N;q,t). \]

\subsection{Jack Symmetric and Heckman-Opdam Hypergeometric Functions}
We describe two degenerations of the Macdonald symmetric polynomials. For $\theta > 0$, the Jack symmetric polynomials are
\[ J_\lambda(x_1,\ldots,x_N;\theta) := \lim_{q \to 1} P_\lambda(x_1,\ldots,x_N;q,q^\theta), \quad \lambda \in \Y_N, \]
see \cite{Mac}*{Chapter VI, Section 10}. Let $\wh{J}_\lambda$ denote the limit obtained by replacing $P_\lambda$ with $\wh{P}_\lambda$.

For $\theta > 0$ and $\br = (r_1 > \cdots > r_N > 0 = r_{N+1} = \cdots = r_M)$, define the \emph{Heckman-Opdam hypergeometric function} (for Type A root systems \cites{HO87I,HO87II,HS,Opd88III,Opd88IV}) by
\begin{align*}
t = q^\theta, \quad q &= \exp(-\e), \quad \lambda = \lfloor \e^{-1} (r_1,\ldots,r_N) \rfloor, ~~~ x_i = \exp( \e z_i), \\
\cF_{\br}(z_1,\ldots,z_M;\theta) & := \lim_{\e \to 0} \e^{\theta N(N-1)/2 + N(M-N)} P_\lambda(x_1,\ldots,x_M;q,t),
\end{align*}
and its renormalizations
\begin{align*}
\wt{\cF}_{\br}(z_1,\ldots,z_M;\theta) & := \lim_{\e \to 0} \e^{N(\theta-1) + \theta N(N-1)/2 + N(M-N)} Q_\lambda(x_1,\ldots,x_M;q,t), \\
\wh{\cF}_{\br}(z_1,\ldots,z_M;\theta) & := \lim_{\e \to 0} \wh{P}_\lambda(x_1,\ldots,x_M;q,t),
\end{align*}
see \cite{BG1}*{Propositions 6.4 and 6.5} for details. We have
\[ \wh{\cF}_{\br}(z_1,\ldots,z_N) = \frac{\cF_{\br}(z_1,\ldots,z_N)}{\cF_{\br}(0,-\theta,-2\theta,\ldots,-(N-1)\theta)} \]
which extends the definition of $\wh{\cF}_{\br}$ for $\br = (r_1 \ge \cdots \ge r_N)$.

Let
\[ \cR^N := \{\br \in \R_{\ge 0}^N: r_1 \ge \cdots \ge r_N \}, \quad \cU^N := \cR^N \cap [0,1]^N. \]
We consider $\cR^N$ as a subset of $\cR^M$ for $M \ge N$ by identifying $\vec{r} \in \cR^N$ with $(r_1,\ldots,r_N,0,\ldots,0) \in \cR^M$.

The identity (\ref{eq:cauchy1}) implies
\begin{align} \label{eq:cauchy2}
\begin{multlined}
\int_{\cR^{\min(N,M)}} \wt{\cF}_{\br}(a_1,\ldots,a_N;\theta) \cF_{\br}(b_1,\ldots,b_M) \prod_{i=1}^{\min(N,M)} dr_i \\
= \prod_{i=1}^N \prod_{j=1}^M \frac{\Gamma(-a_i - b_j)}{\Gamma(\theta - a_i - b_j)}
=: H((a_1,\ldots,a_N),(b_1,\ldots,b_M);\theta)
\end{multlined}
\end{align}
for $a_1,\ldots,a_N,b_1,\ldots,b_M$ such that $\Re(a_i + b_j) < 0$ for all $i \in [[1,N]],j \in [[1,M]]$ (see \cite{BG1}*{Proposition 6.6}). The variable-index symmetry (\ref{eq:varind1}) implies
\begin{align} \label{eq:varind2}
\wh{\cF}_{\br}(-\lambda_1 - (N-1)\theta, -\lambda_2 - (N-2)\theta,\ldots, -\lambda_N;\theta) = \wh{J}_\lambda(\exp(-r_1),\exp(-r_2),\ldots,\exp(-r_N);\theta)
\end{align}
for $\br \in \cR^N$ and $\lambda \in \Y_N$ (see \cite{GM}*{Section 2}).

We also have
\begin{align*}
\wh{J}_\mu(x_1,\ldots,x_N;\theta) \wh{J}_\nu(x_1,\ldots,x_N;\theta) &= \sum_\lambda c_{\mu\nu}^\lambda(\wh{J};\theta) \wh{J}_\lambda(x_1,\ldots,x_N;\theta) \\
\wh{\cF}_{\vec{\ell}}(z_1,\ldots,z_N;\theta) \wh{\cF}_{\vec{r}}(z_1,\ldots,z_N;\theta) &= \int_{\vec{s}} c_{\vec{\ell},\vec{r}}^{\vec{s}}(\wh{\cF};\theta) \wh{\cF}_{\vec{s}}(z_1,\ldots,z_N;\theta) 
\end{align*}
where
\begin{align} \label{eq:LRconv}
c_{\mu\nu}^\lambda(\wh{P};q,t) \to c_{\mu\nu}^\lambda(\wh{J};\theta), \quad t = q^\theta \to 1
\end{align}
and $c_{\vec{\ell},\vec{r}}^{\vec{s}}(\wh{\cF};\theta)$ is a (possibly signed) measure in $\vec{s}$ with total mass $1$ supported in the set of $\vec{s} \in \cR^N$ satisfying
\[ s_1 + \cdots + s_N = (r_1 + \ell_1) + \cdots (r_N + \ell_N), \quad \quad r_N + \ell_N \le s_i \le r_1 + \ell_1, \quad 1 \le i \le N, \]
see \cite{GM}*{Section 2}. For $\theta = 1/2,1,2$, the measures $c_{\vec{\ell},\vec{r}}^{\vec{s}}$ are nonnegative and therefore probability measures.

\section{\texorpdfstring{$\beta$}{beta}-Jacobi Product Process} \label{sec:process}

In this section, we describe the connection between the special functions from \Cref{sec:background} and $\beta$-Jacobi product processes. The main result of this section (\Cref{thm:mactoprod}) realizes $\beta$-Jacobi product processes as limits of \emph{Macdonald processes} --- Markov chains with distributions exhibiting a special structure in terms of Macdonald symmetric functions. We first find a description for the transition kernels of $\beta$-Jacobi product processes for $\beta = 1,2,4$, then extrapolate to arbitrary $\beta > 0$. The results in this section imply \Cref{thm:existence}, see \Cref{prop:bmarkov} and the following discussion.

Let $\alpha > 0$ and $M,N \in \Z_+$ throughout this section.

\begin{definition}
Fix $\beta \in \{1,2,4\}$. Suppose $X,Y$ are random, independent, $N\times N$ random matrices with right $\U_\beta$-invariant distributions. Let $\bx,\by \in \cR^N$ be the respective random squared singular values of $X,Y$ and
\[ \bx \boxtimes_\beta \by \]
denote the squared singular values of $XY$.
\end{definition}

\begin{definition} 
Let $\theta > 0$. Denote by $\P_\theta^{\alpha,M,N}$ the measure on $N$-particles
\[ \bx = (\underbrace{1,~\ldots,~1,}_{N - \min(M,N)} x_1,\ldots,x_{\min(M,N)}) \in \cU^N \]
with density proportional to
\[ \prod_{1 \le i < j \le \min(M,N)} |x_i - x_j|^{2\theta} \prod_{i=1}^{\min(M,N)} x_i^{\theta \alpha - 1} (1 - x_i)^{\theta|M-N| + \theta - 1} \, dx_i. \]
\end{definition}

This is the $\beta$-Jacobi ensemble with $\beta = 2\theta$.

\begin{remark}
The $\beta$-Jacobi ensemble can be defined for $M \in \R_+$ if $M \ge N$. Although the discussion in this section requires $M \in \Z^+$, the formulas in \Cref{sec:moments} can be extended to real $M \ge N$ by analytic continuation.
\end{remark}

For $\theta = 1/2,1,2$, this is the distribution of the squared singular values of a truncated Haar $\U_{2\theta}$ matrix.

\begin{proposition}[{\cite{For10}*{Proposition 3.8.2}}] \label{prop:U_beta}
Let $\beta \in \{1,2,4\}$, $L, N',N \in \Z_+$ such that $L \ge N' \ge N$, and $U$ be random Haar $\U_\beta(L)$ distributed. If $X$ is an $N'\times N$ submatrix of $U$ and
\[ \alpha = N' - N + 1, \quad M = L - N', \]
then the squared singular values of $X$ are $\P_{\beta/2}^{\alpha,M,N}$ distributed.
\end{proposition}

\begin{remark}
At first glance, the realization of the $\beta$-Jacobi ensemble as the squared singular values of a \emph{rectangular} truncated Haar $\U_\beta$-matrix may appear incompatible with our definition of $\bx \boxtimes_\beta \by$ involving \emph{square} matrices $X$ and $Y$. However, it can be shown that these two matrix interpretations are compatible. More specifically, if the distributions of $\bx^{(1)},\ldots,\bx^{(T)}$ are $\beta$-Jacobi ensembles with certain parameters, then the distribution of $\bx^{(1)} \boxtimes \cdots \boxtimes \bx^{(T)}$ can be identified with that of the squared singular values of a product of $T$ rectangular, truncated Haar $\U_\beta$ matrices. We provide the details of this identification in \Cref{sec:matrix}.
\end{remark}

The density of a $\beta$-Jacobi ensemble can also be given in terms of Heckman-Opdam hypergeometric functions. The following proposition is a consequence of a general convergence statement \cite{BG1}*{Theorem 2.8} for $\beta$-Jacobi \emph{corners} processes --- multilevel extensions of $\beta$-Jacobi ensembles.

We adopt the following notation for brevity: for any $\mu \in \Y_N$, set
\begin{align*}
\rho_\mu^N := \rho_\mu^N(q,t) &:= (q^{\mu_1} t^{N - 1}, q^{\mu_2} t^{N - 2},\ldots,q^{\mu_N}),\\
\varrho_\mu^N := \varrho_\mu^N(\theta) &:= -(\mu_1 + \theta(N-1),\mu_2 + \theta(N-2),\ldots, \mu_N), \\
(a_1,\ldots,a_N) + c & := (a_1 + c,\ldots,a_N + c) \quad \mbox{for} \quad a_1,\ldots,a_N,c \in \R,
\end{align*}
where the dependence on $(q,t)$ and $\theta$ will be clear from context.

\begin{proposition} \label{prop:mac_to_jacobi}
Let $\theta > 0$. Let $\e > 0$ be a small parameter with $t = q^\theta, q = e^{-\e}$, and $\lambda$ be a random element of $\Y_N$ with
\[ \prob(\lambda = \mu) = \frac{1}{\Pi(\rho_0^N, t^\alpha \rho_0^M;q,t)} P_\mu(\rho_0^N;q,t) Q_\mu(t^\alpha\rho_0^M;q,t). \]
If $\by \sim \P_\theta^{\alpha,M,N}$, then $q^\lambda \to \by$ in distribution as $\e \to 0$. In particular, the density of $\br = -\log \by$ is
\[ \frac{1}{H(\varrho_0^N,\varrho_0^M - \alpha\theta;\theta)}\cF_{\br}(\varrho_0^N;\theta) \wt{\cF}_{\br}(\varrho_0^M - \alpha \theta;\theta) \, d\br. \]
\end{proposition}

\begin{definition}
Let $\beta \in \{1,2,4\}$ and $\bx^{(1)},\bx^{(2)},\ldots$ be independent, random elements of $\cR^N$. Define the $\beta$-\emph{product process} on $(\bx^{(T)})_{T \in \Z_+}$ to be the random sequence $(\by^{(T)})_{T \in \Z_+}$ where $\by^{(1)} := \bx^{(1)}$ and
\[ \by^{(T+1)} := \by^{(T)} \boxtimes_\beta \bx^{(T+1)}, \quad \quad T \in \Z_+. \]
\end{definition}

By the independence of $\bx^{(1)},\bx^{(2)},\ldots$, the $\beta$-product process is a Markov process in discrete time $T$. We compute the Markov transition probabilities of this process for $\bx^{(T)}$ distributed as a $\beta$-Jacobi ensemble.

\begin{proposition} \label{prop:markov}
Let $\theta \in \{1/2,1,2\}$. If $\bx \sim \P_\theta^{\alpha,M,N}$, $\by \in \cU^N$ and $\bz = \by \boxtimes_{2\theta} \bx$, then
\begin{align} \label{eq:Jmom}
\E \wh{J}_\kappa(\bz;\theta) = \wh{J}_\kappa(\by;\theta) \frac{H(\varrho_\kappa^N,\varrho_0^M - \alpha\theta)}{H(\varrho_0^N,\varrho_0^M-\alpha\theta)}, \quad \kappa \in \Y_N.
\end{align}
The distribution of $\bz$ is determined by (\ref{eq:Jmom}).
\end{proposition}
\begin{proof}
If $\bl,\br \in \cR^N$, then by \cite{GM}*{Proposition 2.2} the probability measure of the random vector $\bs \in \cR^N$ defined by
\[ \exp(-\bs) = \exp(-\bl) \boxtimes_{2\theta} \exp(-\br) \]
is given by $c_{\bl,\br}^{\bs}(\wh{\cF};\theta)$. If we let $\br \in \cR^N$ be random such that $\exp(-\br) \sim \P_\theta^{\alpha,M,N}$, then for any $\kappa \in \Y_N$,
\begin{align*}
\E \wh{J}_\kappa(\exp(-\bs);\theta) &= \int_{\br} \frac{1}{H(\varrho_0^N, \varrho_0^M - \alpha\theta)} \cF_{\br}(\varrho_0^N;\theta) \wt{\cF}_{\br}(\varrho_0^M - \alpha\theta;\theta) \int_{\bs}  c_{\bl,\br}^{\bs}(\wh{\cF};\theta) \wh{J}_\kappa(\exp(-\bs);\theta) \, d\br \\
&= \frac{1}{H(\varrho_0^N, \varrho_0^M - \alpha\theta)} \int_{\br} \cF_{\br}(\varrho_0^N;\theta) \wt{\cF}_{\br}(\varrho_0^M - \alpha\theta;\theta) \wh{\cF}_{\bl}(\varrho_\kappa^N;\theta) \wh{\cF}_{\br}(\varrho_\kappa^N;\theta) \, d\br \\
&= \frac{\wh{J}_\kappa(\exp(-\bl);\theta)}{H(\varrho_0^N, \varrho_0^M - \alpha\theta)} \int_{\br} \cF_{\br}(\varrho_\kappa^N;\theta) \wt{\cF}_{\br}(\varrho_0^M - \alpha\theta;\theta) \, d\br \\
&= \wh{J}_\kappa(\exp(-\bl);\theta) \frac{H(\varrho_\kappa^N, \varrho_0^M - \alpha\theta)}{H(\varrho_0^N, \varrho_0^M - \alpha\theta)}
\end{align*}
where the first equality uses \Cref{prop:mac_to_jacobi}, the second and third equalities use (\ref{eq:varind2}), and the fourth equality uses (\ref{eq:cauchy2}). Taking $\by = \exp(-\bl)$ and $\bz = \exp(-\bs)$ proves (\ref{eq:Jmom}). Since the Jack symmetric functions in $N$-variables form a basis for the space of symmetric polynomials in $N$-variables and since $\bz$ is supported in $\cU^N$, (\ref{eq:Jmom}) determines the distribution of $\bz$.
\end{proof}

The transition probabilities of Proposition \ref{prop:markov} can be extended to arbitrary $\beta > 0$. We detail this extrapolation below. The idea is seeing the Markov kernels of Proposition \ref{prop:markov} as limits of a family of kernels derived from Macdonald symmetric functions.

\begin{definition}
Let $0 < q,t < 1$. Define the kernel
\[ K_{q,t}^{\alpha,M,N}(\mu,\lambda) := \frac{1}{\Pi(\rho_0^N,t^\alpha \rho_0^M;q,t)} \frac{P_\lambda(\rho_0^N;q,t)}{P_\mu(\rho_0^N;q,t)} Q_{\lambda/\mu}(t^\alpha \rho_0^M;q,t), \quad \quad \lambda,\mu \in \Y_N \]
\end{definition}

\begin{lemma} \label{lem:Pmom}
For each $\kappa, \mu \in \Y_N$, we have
\begin{align}
\sum_{\lambda \in \Y} \wh{P}_\kappa(\rho_\lambda^N;q,t) K_{q,t}^{\alpha,M,N}(\mu,\lambda) &= \wh{P}_\kappa(\rho_\mu^N;q,t) \frac{\Pi(\rho_\kappa^N,t^\alpha \rho_0^M;q,t)}{\Pi(\rho_0^N, t^\alpha \rho_0^M;q,t)}.
\end{align}
In particular, $K_{q,t}^{\alpha,M,N}(\mu,\cdot)$ defines a probability distribution.
\end{lemma}
\begin{proof}
We contract the notation by dropping the $(q,t)$ from $\Pi, P, \wh{P},Q$. We have
\begin{align*}
\sum_{\lambda \in \Y} \wh{P}_\kappa(\rho_\lambda^N) K_{q,t}^{\alpha,M,N}(\mu,\lambda) &= \frac{1}{P_{\mu}(\rho_0^N) \Pi(\rho_0^N, t^\alpha \rho_0^M)} \sum_{\lambda \in \Y} \wh{P}_\lambda(\rho_\kappa^N) P_\lambda(\rho_0^N) Q_{\lambda/\mu}(t^\alpha \rho_0^M) \\
&= \frac{1}{P_{\mu}(\rho_0^N) \Pi(\rho_0^N, t^\alpha \rho_0^M)} \sum_{\lambda \in \Y} P_\lambda(\rho_\kappa^N) Q_{\lambda/\mu}(t^\alpha \rho_0^M) \\
&= \wh{P}_\mu(\rho_\kappa^N) \frac{\Pi(\rho_\kappa^N,t^\alpha \rho_0^M)}{\Pi(\rho_0^N, t^\alpha \rho_0^M)}
\end{align*}
where we use \eqref{eq:varind1} in the first equality and \eqref{eq:flipid} in the final equality. By taking $\kappa = (0)$, we see that $K_{q,t}^{\alpha,M,N}(\mu,\cdot)$ defines a probability distribution. By \eqref{eq:varind1}, the lemma follows.
\end{proof}

\begin{proposition} \label{prop:bmarkov}
Let $\theta > 0$. Suppose $\bv \in \cU^N$, and $\e > 0$ is a small parameter with $q = e^{-\e}$, $t = q^\theta$, and $\mu = \mu(\e) \in \Y_N$ such that $q^\mu \to \bv$ as $\e \to 0$. If $\lambda \sim K_{q,t}^{\alpha,M,N}(\mu,\cdot)$, then as $\e \to 0$, $q^\lambda$ converges in distribution to a probability measure $\cK_\theta^{\alpha,M,N}(\bv,d\bu)$ on $\cU^N$ determined by
\begin{align} \label{eq:bJmom}
\int_{\cU^N} \wh{J}_\kappa(\bu;\theta) \cK_\theta^{\alpha,M,N}(\bv,d\bu) =  \wh{J}_\kappa(\bv;\theta) \frac{H(\varrho_\kappa^N,\varrho_0^M - \alpha\theta;\theta)}{H(\varrho_0^N,\varrho_0^M - \alpha\theta;\theta)}, \quad \kappa \in \Y_N.
\end{align}
\end{proposition}
\begin{proof}
By Lemma \ref{lem:Pmom}, for any $\kappa \in \Y_N$ we have the convergence
\[ \lim_{\e \to 0} \sum_{\upsilon \in \exp(-\e \Y_N)} \wh{P}_\kappa \left(\upsilon_1 t^{N-1},\upsilon_2 t^{N-2},\ldots,\upsilon_N;q,t \right) K_{q,t}^{\alpha,M,N}\left(\mu,-\e^{-1} \log \upsilon\right) = \wh{J}_\kappa(\bv;\theta) \frac{H(\varrho_\kappa^N,\varrho_0^M - \alpha\theta;\theta)}{H(\varrho_0^N,\varrho_0^M - \alpha\theta;\theta)}, \]
and the uniform convergence
\[ \lim_{\e \to 0} \wh{P}_\kappa\left(u_1 t^{N-1},u_2 t^{N-2},\ldots, u_N; q,t \right) = \wh{J}_\kappa(u_1,u_2,\ldots, u_N; \theta ) \]
over $(u_1,\ldots,u_N) \in \cU^N$. Therefore
\[ \lim_{\e \to 0} \sum_{\upsilon \in \exp(-\e \Y_N)} \wh{J}_\kappa (\upsilon_1,\ldots,\upsilon_N;\theta) K_{q,t}^{\alpha,M,N}\left(\mu,-\e^{-1} \log \upsilon\right) = \wh{J}_\kappa(\bv;\theta) \frac{H(\varrho_\kappa^N,\varrho_0^M - \alpha\theta;\theta)}{H(\varrho_0^N,\varrho_0^M - \alpha\theta;\theta)}. \]
Since $\cU^N$ is compact and the Jack symmetric functions are dense in $C(\cU^N,\R)$, the desired convergence and (\ref{eq:bJmom}) follow. The fact that $\cK_\theta^{\alpha,M,N}(\bv,d\bu)$ is a probability measure comes from taking $\kappa = (0)$ above.
\end{proof}

Using the notation $\bv \boxtimes_{2\theta} \mathbf{w}$ to denote the random vector distributed as $\cK_\theta^{\alpha,M,N}(\bv,d\bu)$ where $\mathbf{w} \sim \P_\theta^{\alpha,M,N}$, we can reexpress \eqref{eq:bJmom} as
\[ \E \wh{J}_\kappa(\bv \boxtimes_{2\theta} \mathbf{w};\theta) = \wh{J}_\kappa(\bv;\theta) \E \wh{J}_\kappa(\mathbf{w};\theta). \]
Indeed, we have
\begin{align*}
\E \wh{J}_\kappa(\mathbf{w};\theta) &= \frac{1}{H(\varrho_0^N,\varrho_0^M - \alpha\theta;\theta)} \int \wh{J}_\kappa(\exp(-\br);\theta) \cF_{\br}(\varrho_0^N;\theta) \wt{\cF}_{\br}(\varrho_0^M - \alpha \theta;\theta) \\
&= \frac{1}{H(\varrho_0^N,\varrho_0^M - \alpha\theta;\theta)} \int \cF_{\vec{r}}(\varrho_\kappa^N;\theta) \wt{\cF}_{\br}(\varrho_0^M - \alpha \theta;\theta) = \frac{H(\varrho_\kappa^N,\varrho_0^M - \alpha\theta;\theta)}{H(\varrho_0^N,\varrho_0^M - \alpha\theta;\theta)}
\end{align*}
where we use \Cref{prop:mac_to_jacobi} for the first equality, \eqref{eq:varind2} for the second, and \eqref{eq:cauchy2} for the third. Therefore, we see that \Cref{prop:bmarkov} implies \Cref{thm:existence}.

For the remainder of this section, fix $\bA := (\alpha_T,M_T)_{T \in \Z_+} \in (\R_+ \times \Z_+)^{\Z_+}$.

\begin{definition}
Let $\theta > 0$. The \emph{$\beta$-Jacobi product process with parameter $\bA$} is a Markov chain $(\by^{(T)})_{T \in \Z_+}$ with state space $\cU^N$ such that (i) $\by^{(1)} \sim \P_\theta^{\alpha_1,M_1,N}$ and (ii) $\cK_\theta^{\alpha_T,M_T,N}(\bv,d\bu)$ is the Markov kernel from $\by^{(T-1)}$ to $\by^{(T)}$. We denote the associated probability measure by $\P_\theta^{\bA,N}$.
\end{definition}

We refer to Markov chains of the form above as \emph{$\beta$-Jacobi product processes} where $\beta = 2\theta$. Informally, we may view $\by^{(T)}$ as satisfying
\[ \by^{(T+1)} = \by^{(T)} \boxtimes_\beta \bx^{(T)} \]
where $\bx^{(T)} \sim \P_\theta^{\alpha_T,M_T,N}$, $\by^{(1)} = \bx^{(1)}$ and $\boxtimes_\beta$ is an extension of the operation for $\beta = 1,2,4$.

\begin{definition} \label{def:iaMp}
Let $\P_{q,t}^{\bA,N}$ be the measure on sequences $(\lambda^T)_{T \in \Z_+} \in (\Y_N)^\infty$ where
\begin{align*}
\P_{q,t}^{\bA,N}(\lambda^1 = \mu) &= \frac{1}{\Pi(\rho_0^N,t^{\alpha_1} \rho_0^{M_1})} P_\mu(\rho_0^N;q,t) Q_\mu(t^{\alpha_1} \rho_0^{M_1};q,t), \\
\P_{q,t}^{\bA,N}(\lambda^T = \lambda | \lambda^{T-1} = \mu) &= K_{q,t}^{\alpha_T,M_T,N}(\mu,\lambda), \quad T > 1.
\end{align*}
\end{definition}

The measure $\P_{q,t}^{\bA,N}$ is an example of a so-called \emph{(infinite) ascending Macdonald process} \cite{BG1}*{Section 2.1}.

\begin{theorem} \label{thm:mactoprod}
Let $\theta > 0$. Let $\e > 0$ be a small parameter with $q = e^{-\e}$, $t = q^\theta$, $(\lambda^T := \lambda^T(\e))_{T \in \Z_+} \sim \P_{q,t}^{\bA,N}$. Then the finite dimensional distributions of $(q^{\lambda^T} )_{T \in \Z_+}$ converge weakly to those of $\P_\theta^{\bA,N}$ as $\e \to 0$.
\end{theorem}
\begin{proof}
We drop the $(q,t)$ and $\theta$ in our notation, though all limits are as $\e \to 0$ with $q = e^{-\e}$, $t = q^\theta$. Let $(\by^{(T)})_{T\in\Z_+} \sim \P_\theta^{\bA,N}$. It suffices to show that
\begin{align*}
\E \wh{P}_{\kappa^1}(\rho_{\lambda^1}^N) \cdots \wh{P}_{\kappa^T}(\rho_{\lambda^T}^N) \to \E \wh{J}_{\kappa^1}(\by^{(1)}) \cdots \wh{J}_{\kappa^T}(\by^{(T)})
\end{align*}
for any $T \ge 1$. We induct on $T$, observing that the $T = 1$ case follows from \Cref{prop:mac_to_jacobi}. If we assume the $T-1$ case, then
\begin{align*}
\E \wh{P}_{\kappa^1}(\rho_{\lambda^1}^N) \cdots \wh{P}_{\kappa^T}(\rho_{\lambda^T}^T) &= \frac{\Pi(\rho_{\kappa^T}^N, t^{\alpha_T} \rho_0^{M_T})}{\Pi(\rho_0^N, t^{\alpha_T} \rho_0^{M_T})} \E \wh{P}_{\kappa^1}(\rho_{\lambda^1}^N) \cdots \wh{P}_{\kappa^{T-1}}(\rho_{\lambda^{T-1}}) \wh{P}_{\kappa^T}(\rho_{\lambda^{T-1}}) \\
&= \frac{\Pi(\rho_{\kappa^T}^N, t^{\alpha_T} \rho_0^{M_T})}{\Pi(\rho_0^N, t^{\alpha_T} \rho_0^{M_T})} \sum_\mu c^\mu_{\kappa^{T-1} \kappa^T}(\wh{P}) \E \wh{P}_{\kappa^1}(\rho_{\lambda^1}^N) \cdots \wh{P}_{\kappa^{T-2}} (\rho_{\lambda^{T-2}}) \wh{P}_\mu(\rho_{\lambda^{T-1}})
\end{align*}
where the first equality follows from Lemma \ref{lem:Pmom}. The latter converges, by our induction hypothesis and (\ref{eq:LRconv}), to
\begin{align*}
\frac{H(\varrho_{\kappa^T}^N, \varrho_0^{M_T} - \alpha_T \theta)}{H(\varrho_0^N, \varrho_0^{M_T} - \alpha_T \theta)} \sum_\mu c^\mu_{\kappa^{T-1}\kappa^T}(\wh{J}) \E \wh{J}_{\kappa^1}(\by^{(1)}) \cdots \wh{J}_{\kappa^{T-2}} (\by^{(T-2)}) \wh{J}_\mu(\by^{(T-1)})
\end{align*}
which can be written as
\begin{align*}
\frac{H(\varrho_{\kappa^T}^N, \varrho_0^{M_T} - \alpha_T \theta)}{H(\varrho_0^N, \varrho_0^{M_T} - \alpha_T \theta)} \E \wh{J}_{\kappa^1}(\by^{(1)}) \cdots \wh{J}_{\kappa^{T-2}} (\by^{(T-2)}) \wh{J}_{\kappa^{T-1}}(\by^{(T-1)}) \wh{J}_{\kappa^T}(\by^{(T-1)}) = \E \wh{J}_{\kappa^1}(\by^{(1)}) \cdots \wh{J}_{\kappa^T}(\by^{(T)})
\end{align*}
by Proposition \ref{prop:bmarkov}.
\end{proof}

\section{Observables} \label{sec:moments}
The main results of this section are formulas for the joint moments of $\beta$-Jacobi product processes. We provide formulas for general $\beta > 0$ and more convenient formulas for $\beta = 2$.

Fix $N \in \Z_+$ and $\bA := (\alpha_T,M_T)_{T \in \Z_+} \in (\R_+ \times \Z_+)^{\Z_+}$ throughout this section. Given $U = (u_1,\ldots,u_k)$, $V = (v_1,\ldots,v_\ell)$, define
\begin{align*}
\cA_T(U) := \cA_T(U;\theta) &:= \prod_{i=1}^k \frac{u_i}{u_i + \theta N} \prod_{\tau = 1}^T \frac{u_i - \theta(\alpha_\tau - 1)}{u_i - \theta(\alpha_\tau + M_\tau - 1)}, \\
\cB(U) := \cB(U;\theta) &:= \frac{1}{(u_2 - u_1 + 1 - \theta) \cdots (u_k - u_{k-1} + 1 - \theta)} \prod_{1 \le i < j \le k} \frac{(u_j - u_i)(u_j - u_i + 1 - \theta)}{(u_j - u_i + 1)(u_j - u_i - \theta)}, \\
\cC(U,V) := \cC(U,V;\theta) &:= \prod_{i=1}^k \prod_{j=1}^\ell \frac{(v_j - u_i)(v_j - u_i + 1 - \theta)}{(v_j - u_i + 1)(v_j - u_i - \theta)}.
\end{align*}

\begin{definition}
Given $(x_1,\ldots,x_N) \in \R^N$, define for any $k > 0$
\[ \fP_k(x_1,\ldots,x_N) = \sum_{i=1}^N x_i^k. \]
\end{definition}

\begin{theorem} \label{thm:obs}
Let $\theta > 0$ and $(\by^{(T)})_{T \in \Z_+} \sim \P_\theta^{\bA,N}$. If $T_1 \ge \cdots \ge T_m > 0$ and $k_1,\ldots,k_m > 0$ are integers, then
\[ \E\left[ \prod_{i=1}^m \fP_{k_i}(\by^{(T_i)}) \right] = \frac{(-\theta)^{-m}}{(2\pi \bi)^{k_1 + \cdots + k_m}} \oint \cdots \oint \prod_{1 \le i < i' \le m} \cC(U_i,U_{i'}) \prod_{i=1}^m \cB(U_i) \cA_{T_i}(U_i) dU_i \]
where $U_i = (u_{i,1},\ldots,u_{i,k_i})$,
\begin{enumerate}
    \item the $u_{i,j}$-contour $\fU_{i,j}$ is positively oriented around the pole at $-\theta N$ and does not enclose $\theta(\alpha_\tau + M_\tau - 1)$ for $1 \le \tau \le T_i$;
    \item whenever $(i,j) < (i',j')$ in lexicographical order, $\fU_{i,j}$ is enclosed by $\fU_{i',j'} - \theta, \fU_{i',j'} + 1$;
\end{enumerate}
given that such contours exist.
\end{theorem}

\begin{remark}
The existence of the contours is guaranteed for $N$ large. Since our applications are for $N$ large, the question of existence is not a hindrance.
\end{remark}

\begin{theorem} \label{thm:obs_cpx}
Let $(\by^{(T)})_{T \in \Z_+} \sim \P_1^{\bA,N}$. If $T_1 \ge \cdots \ge T_m > 0$ are integers and $c_1,\ldots,c_m > 0$ are real, then
\begin{align*}
\begin{multlined}
\E\left[ \prod_{i=1}^m \fP_{c_i}(\by^{(T_i)}) \right] = \frac{\prod_{i=1}^m (-c_i)^{-1}}{(2\pi\bi)^m} \oint \cdots \oint \prod_{1 \le i < j \le m} \frac{(u_j - u_i)(u_j + c_j - u_i - c_i)}{(u_j - u_i - c_i)(u_j + c_j - u_i)} \\
\times \prod_{i=1}^m \left( \prod_{\ell=1}^N \frac{u_i + \ell - 1}{u_i + c_i + \ell - 1} \cdot \prod_{\tau = 1}^{T_i} \prod_{\ell=1}^{M_\tau} \frac{u_i + c_i - \alpha_\tau - \ell + 1}{u_i - \alpha_\tau - \ell + 1} \right) du_i
\end{multlined}
\end{align*}
where the $u_i$-contour $\fU_i$ is positively oriented around $\{-c_i- \ell+1\}_{\ell=1}^N$ but does not enclose $\alpha_\tau+\ell-1$ for $1 \le \ell \le M_\tau$, $1 \le \tau \le T_i$, and is enclosed by $\fU_j - c_i, \fU_j + c_j$ for $j > i$; given that such contours exist.
\end{theorem}

\begin{remark} \label{rem:extension}
We note that \Cref{thm:obs} holds for arbitrary $\theta > 0$, but is restricted to taking $\fP_{k_i}$ where $k_i$ are positive \emph{integers}. Moreover, the contours have dimension $k_1 + \cdots + k_m$. In contrast, the contours in \Cref{thm:obs_cpx} for $\theta = 1$ have dimension $m$, and $k_i = c_i$ can be an arbitrary positive real. In this article, we come to these two formulas from seemingly different approaches. However, it was pointed out by E. Dimitrov that the $\theta > 0$ implies the $\theta = 1$ formulas. By residue expansion and combinatorics, the higher dimensional contour integral formulas reduce to $m$-dimensional contour integral formulas when $\theta = 1$. The reason we can take $c_1,\ldots,c_m > 0$ arbitrary follows from analytic continuation. 
\end{remark}

The remainder of this section is devoted to the proofs of \Cref{thm:obs,thm:obs_cpx}. The starting point is a set of contour integral formulas for exponential moments of Macdonald processes, obtained in \cite{GZ} and further generalized in \cite{Ahn1}. For the $\theta = 1$ case, we use specialized formulas which are derived in \Cref{sec:obs}. \Cref{thm:obs,thm:obs_cpx} follow by applying \Cref{thm:mactoprod} and taking suitable limit transitions of these formulas.

\subsection{Formulas for Macdonald Processes}
We state and prove discrete analogues of \Cref{thm:obs,thm:obs_cpx} in the setting of Macdonald processes (recall \Cref{def:iaMp}).

\begin{definition} \label{def:p}
For $\lambda \in \Y_n$ define
\[ \fp_k(\lambda;q,t) := (1 - t^{-k}) \sum_{i=1}^n q^{k\lambda_i} t^{k(-i+1)} + t^{-kn} \]
for integers $k > 0$. We write $\fp_k(\lambda)$ if $q,t$ is clear from context. For $t > 0$, define
\[ \boldsymbol{p}_t(\lambda) := (1 - t^{-1})\sum_{i=1}^n t^{\lambda_i - i + 1} + t^{-n}. \]
Observe that these definitions are independent of $n$ as long as $\ell(\lambda) \le n$. 
\end{definition}

If we represent an ordered $k$-tuple of variables $(z_1,\ldots,z_k)$ by $Z$, then $Z^{-1} := (z_1^{-1},\ldots,z_k^{-1})$, $qZ := (qz_1,\ldots,qz_k)$, $dZ := dz_1 \cdots dz_k$. Given $Z = (z_1,\ldots,z_k)$ and $W = (w_1,\ldots,w_\ell)$, let 
\begin{align*}
\fA_T(Z) := \fA_T(Z;q,t) &:= \prod_{i=1}^k \frac{z_i - 1}{z_i - t^N} \prod_{\tau = 1}^T \frac{1 - t^{\alpha_\tau - 1}z_i}{1 - t^{\alpha_\tau + M_\tau - 1}z_i}, \\
\fB(Z) := \fB(Z;q,t) &:= \frac{\sum_{i=1}^k \frac{1}{z_i} \frac{q^{i-1}}{t^{i-1}}}{(z_2 - \frac{q}{t} z_1) \cdots (z_k - \frac{q}{t}z_{k-1})} \prod_{1 \le i < j \le k} \frac{(z_j - z_i)(z_j - \frac{q}{t} z_i)}{(z_j - q z_i)(z_j - \frac{1}{t} z_i)}, \\
\fC(Z,W) := \fC(Z,W;q,t) &:=  \prod_{i=1}^k \prod_{j=1}^\ell \frac{(w_j - z_i)(w_j - \frac{q}{t}z_i)}{(w_j - qz_i)(w_j - \frac{1}{t}z_i)}.
\end{align*}

\begin{proposition} \label{prop:macobs}
Let $(\lambda^T)_{T>0} \sim \P_{q,t}^{\bA,N}$. If $T_1 \ge \cdots \ge T_m > 0$ and $k_1,\ldots,k_m > 0$ are integers, then 
\[ \E \left[ \fp_{k_1}(\lambda^{T_1}) \cdots \fp_{k_m}(\lambda^{T_m}) \right] = \frac{1}{(2\pi \bi)^{k_1 + \cdots + k_m}} \oint \cdots \oint \prod_{1 \le i < i' \le m} \fC(Z_i, Z_{i'}) \prod_{i = 1}^m \fB(Z_i) \fA_{T_i}(Z_i) \, dZ_i \]
where $Z_i = (z_{i,1},\ldots,z_{i,k_i})$,
\begin{enumerate}
    \item the $z_{i,j}$-contour is positively oriented around $0, t^{N}$ but does not enclose $t^{-\alpha_\tau-M_\tau+1}$ for $1 \le \tau \le T_i$;
    \item the contours satisfy $|z_{i,j}| < t|z_{i',j'}|$ for any $(i,j) < (i',j')$ in lexicographical order;
\end{enumerate}
given that such contours exist.
\end{proposition}

\begin{proof}
Choose $T \ge \max(T_1,\ldots,T_m)$. Consider the measure $\fM$ on $\Y^T$ defined by
\[ \fM(\mu^1,\ldots,\mu^T) = \frac{1}{\Pi(\rho^+;\rho^-_1,\ldots,\rho^-_T)} P_{\mu^T}(\rho^+) Q_{\mu^T/\mu^{T-1}}(\rho^-_T) Q_{\mu^{T-1}/\mu^{T-2}}(\rho^-_{T-1}) \cdots Q_{\mu^2/\mu^1}(\rho^-_2) Q_{\mu^1}(\rho^-_1) \]
where $\rho^+ = (a_1,\ldots,a_N), \rho^-_\tau = (b_{\tau,1},\ldots,b_{\tau,M_\tau})$ and $a_i$, $b_{\tau,j} > 0$ are chosen so that the quantity above is summable over $(\mu^1,\ldots,\mu^T) \in \Y^T$. Define
\[ G_\tau(z) = \prod_{i=1}^N \frac{1 - t^{-1}a_i z^{-1}}{1 - a_i z^{-1}} \prod_{\tau' = 1}^\tau \prod_{i=1}^{M_{\tau'}} \frac{1 - b_{\tau',j}z}{1 - tb_{\tau',j}z}. \]

For such a distribution, \cite{Ahn1}*{Theorem 3.8} shows that
\[ \E_{\fM}[ \fp_{k_1}(\mu^{T_1}) \cdots \fp_{k_m}(\mu^{T_m})] = \frac{1}{(2\pi \bi)^{k_1 + \cdots + k_m}} \oint \cdots \oint \prod_{1 \le i < i' \le m} \fC(Z_i,Z_{i'}) \prod_{i = 1}^m \fB(Z_i) G_{T_i}(Z_i) \, dZ_i \]
where $|Z_i| = k_i$, the $z_{i,j}$-contour is positively oriented around all the poles of $G_{T_i}$ among $0,a_1,\ldots,a_N$, but does not contain any poles of $G_{T_i}$ among $b_{\tau,1},\ldots,b_{\tau,M_\tau}$ for $1 \le \tau \le T_i$, and $|z_{i,j}| < t|z_{i',j'}|$ for $(i,j) < (i',j')$ in lexicographical order. Setting $\rho^+ = \rho_0^N$ and $\rho_\tau^- = t^{\alpha_\tau} \rho_0^{M_\tau}$ implies the desired result.
\end{proof}

\begin{proposition} \label{prop:schurobs}
Let $(\lambda^T)_{T>0} \sim \P_{q,q}^{\bA,N}$. If $T_1 \ge \cdots \ge T_m > 0$ are integers and $t_1,\ldots,t_m > 0$ are real, then
\begin{align*}
\begin{multlined}
\E\left[ \prod_{i=1}^m \boldsymbol{p}_{t_i}(\lambda^{T_i}) \right] = \frac{1}{(2\pi\bi)^m} \oint \cdots \oint \prod_{1 \le i < j \le m} \frac{(z_j - z_i)(t_i z_j - t_j z_i)}{(t_i z_j - z_i)(z_j - t_j z_i)} \\
\times \prod_{i=1}^m \left( \prod_{\ell=1}^N \frac{z_i - q^{\ell-1}}{z_i - t_i q^{\ell-1}} \cdot \prod_{\tau = 1}^{T_i} \prod_{\ell=1}^{M_\tau} \frac{1 - t_i^{-1} q^{\alpha_\tau + \ell - 1} z_i}{1 - q^{\alpha_\tau + \ell - 1} z_i}\right) \frac{dz_i}{z_i}
\end{multlined}
\end{align*}
where the $z_i$-contour $\fY_i$ is positively oriented around $0, \{t_i q^{\ell-1}\}_{\ell=1}^N$ but does not enclose $q^{-\alpha_\tau-\ell+1}$ for $1 \le \ell \le M_\tau$, $1 \le \tau \le T_i$, and is encircled by $t_j^{-1} \fY_j$, $t_i \fY_j$ for $j > i$; given that such contours exist.
\end{proposition}

\begin{proof}
For any integers $T_1 \ge \cdots \ge T_m > 0$, we have
\begin{align*}
\P_{q,q}^{\bA}(\lambda^1 = \mu^1, \ldots,\lambda^T = \mu^T) = s_{\mu^T}(\rho_0^N) s_{\mu^T/\mu^{T-1}}(q^{\alpha_T} \rho_0^{M_T}) \cdots  s_{\mu^2/\mu^1}(q^{\alpha_2} \rho_0^{M_2}) s_{\mu^1}(q^{\alpha_1} \rho_0^{M_1})
\end{align*}
where $\rho_0^N := \rho_0^N(q,q) = (1,q,\ldots,q^{N-1})$. The distribution of $(\lambda^1,\ldots,\lambda^T)$ can be described in terms of Schur processes (see \Cref{def:schur_process}). If $\rho = (q^{\alpha_1} \rho_0^{M_1},\ldots,q^{\alpha_T} \rho_0^{M_T})$ and $(\nu^1,\ldots,\nu^{1 + \sum_{i=1}^T M_i}) \sim \mathbb{SP}_{\rho_0^N,\rho}$, then
\[ (\lambda^1,\ldots,\lambda^T) \overset{distr}{=} (\nu^{M_1}, \nu^{M_1 + M_2},\ldots, \nu^{\sum_{i=1}^T M_i}) \]
by \eqref{eq:schur_branch} and \eqref{eq:schur_process}. Thus \Cref{prop:schurobs} can be seen as a special case of \Cref{thm:schur_obs}.
\end{proof}

\subsection{Excision of Poles at Zero}

The proofs of \Cref{thm:obs,thm:obs_cpx} are essentially taking the appropriate limit transitions of \Cref{prop:macobs,prop:schurobs}, but with a preprocessing step which replaces $\fp_k(\lambda^T)$ with $\fp_k(\lambda^T) - t^{-kN}$ and $\boldsymbol{p}_t(\lambda^T)$ with $\boldsymbol{p}_t(\lambda^T) - t^{-N}$ in the expectations. This replacement corresponds to removing the pole at $0$ in the contours.

\begin{proposition} \label{prop:macobs-}
Let $(\lambda^T)_{T>0} \sim P_{q,t}^{\bA,N}$. If $T_1 \ge \cdots \ge T_m > 0$ and $k_1,\ldots,k_m > 0$ are integers, then
\[ \E\left[  \prod_{i=1}^m \left(\fp_{k_i}(\lambda^{T_i}) - t^{-k_iN}\right) \right] = \frac{1}{(2\pi \bi)^{k_1 + \cdots + k_m}} \oint \cdots \oint \prod_{1 \le i < i' \le m} \fC(Z_i, Z_{i'}) \prod_{i = 1}^m \fB(Z_i) \fA_{T_i}(Z_i) \, dZ_i \]
where $Z_i = (z_{i,1},\ldots,z_{i,k_i})$,
\begin{enumerate}
    \item the $z_{i,j}$-contour $\fY_{i,j}$ is positively oriented around the pole at $t^N$ but does not enclose $0$ or $t^{-\alpha_\tau-M_\tau+1}$ for $1 \le \tau \le T_i$;
    \item $\fY_{i,j}$ is enclosed by $t \fY_{i',j'}, q^{-1}\fY_{i',j'}$ for any $(i,j) < (i',j')$ in lexicographical order;
\end{enumerate}
given that such contours exist.
\end{proposition}

\begin{remark}
There is an asymptotic version of \Cref{prop:macobs-} given by \cite{GZ}*{Lemmas 4.14-4.17}; instead of asserting equality it asserts equality after the Jacobi limit. There are overlapping ideas in the proofs of \Cref{prop:macobs-} and the asymptotic version in \cite{GZ}.
\end{remark}

\begin{proposition} \label{prop:schurobs-}
Let $(\lambda^T)_{T>0} \sim \P_{q,q}^{\bA,N}$. If $T_1 \ge \cdots \ge T_m > 0$ are integers and $t_1,\ldots,t_m > 0$ are real, then
\begin{align*}
\begin{multlined}
\E\left[ \prod_{i=1}^m \left(\boldsymbol{p}_{t_i}(\lambda^{T_i}) - t_i^{-N}\right) \right] = \frac{1}{(2\pi\bi)^m} \oint \cdots \oint \prod_{1 \le i < j \le m} \frac{(z_j - z_i)(t_i z_j - t_j z_i)}{(t_i z_j - z_i)(z_j - t_j z_i)} \\
\times \prod_{i=1}^m \left( \prod_{\ell=1}^N \frac{z_i - q^{\ell-1}}{z_i - t_i q^{\ell-1}} \cdot \prod_{\tau = 1}^{T_i} \prod_{\ell=1}^{M_\tau} \frac{1 - t_i^{-1} q^{\alpha_\tau + \ell - 1} z_i}{1 - q^{\alpha_\tau + \ell - 1} z_i}\right) \frac{dz_i}{z_i}
\end{multlined}
\end{align*}
where the $z_i$-contour $\fY_i$ is positively oriented around $\{t_i q^{\ell-1}\}_{\ell=1}^N$ but does not enclose $0$ or $q^{-\alpha_\tau-\ell+1}$ for $1 \le \ell \le M_\tau$, $1 \le \tau \le T_i$, and is enclosed by $t_j^{-1} \fY_j, t_i \fY_j$ for $j > i$; given that such contours exist.
\end{proposition}

\begin{example} \label{ex:m=1}
We first prove the $m = 1$ case of \Cref{prop:macobs-} with $k_1 = k$ and $T_1 = T$. We may change the contours in \Cref{prop:macobs} to obtain the formula
\begin{align} \label{eq:yz}
\E \fp_k(\lambda^T) = \frac{1}{(2\pi \bi)^k} \oint_{\fY_1 \cup \fZ_1} \cdots \oint_{\fY_k \cup \fZ_k} \fB(Z) \fA_T(Z) dZ
\end{align}
where $Z = (z_1,\ldots,z_k)$, $\fY_i \cup \fZ_i$ is the $z_i$-contour such that: (i) $\fY_i$ is enclosed by $t\fY_j, q^{-1}\fY_j$ and $\fZ_i$ is enclosed by $t\fZ_j, q^{-1}\fZ_j$ for $1 \le i < j \le k$, (ii) $\fY_i$ is positively oriented around the pole $t^N$ but does not enclose any poles in  $\{0\}\cup \{t^{-\alpha_\tau-M_\tau+1}\}_{\tau = 1}^T$ and $\fZ_i$ is positively oriented around the pole at $0$ but does not enclose any poles in $\{t^N\} \cup \{t^{-\alpha_\tau-M_\tau+1}\}_{\tau = 1}^T$, (iii) $\fY_k$ and $\fZ_k$ are disjoint from one another.

On the other hand, \Cref{prop:macobs-} asserts that
\begin{align} \label{eq:y}
\E\left[ \fp_k(\lambda^T) - t^{-kN} \right] = \frac{1}{(2\pi \bi)^k} \oint_{\fY_1} \cdots \oint_{\fY_k} \fB(Z) \fA_T(Z) \, dZ.
\end{align}
We show how to go from (\ref{eq:yz}) to (\ref{eq:y}).

\begin{proof}[Proof of (\ref{eq:y})]
Let $\cP(Z)$ be the power set of $\{z_1,\ldots,z_k\}$. For each element $\Upsilon \in \cP(Z)$, let $\Upsilon_i = \fZ_i$ if $z_i \in \Upsilon$ and $\Upsilon_i = \fY_i$ if $z_i \notin \Upsilon$; $\Upsilon$ is exactly the set of $z_i$ such that $\Upsilon_i$ encircles $0$. Let
\[ \fI_\Upsilon = \frac{1}{(2\pi \bi)^k} \oint_{\Upsilon_1} \cdots \oint_{\Upsilon_k} \fB(Z) \fA_T(Z) \, dZ.\]
Expand (\ref{eq:yz}) so the $z_i$-contour is either $\fY_i$ or $\fZ_i$. Thus
\[ \E \fp_k(\lambda^T) = \sum_{\Upsilon \in \cP(Z)} \fI_\Upsilon. \]
Observe that $\fI_{\{z_1,\ldots,z_k\}} = \fA_T(0)^k = t^{-kN}$ by evaluating residues. Also observe that due to the $\fB(Z)$ term, $\fI_\Upsilon = 0$ unless $\Upsilon$ is of the form $\{z_{r+1},\ldots,z_{r+d}\}$. Identify $\Upsilon = \{z_{r+1},\ldots,z_{r+d}\}$ with $(r,d)$ so that
\begin{align} \label{eq:YZexp}
\E \fp_k(\lambda^T) = t^{-kN} + \fI_\emptyset + \sum_{d=1}^{k-1} \sum_{r=0}^{k-d} \fI_{(r,d)}.
\end{align}
Consider the following three cases for $(r,d) = \Upsilon = \{z_{r+1},\ldots,z_{r+d}\}$:
\begin{enumerate}[(I)]
    \item $r = 0$. Evaluating the residues at $0$ for $z_1,\ldots,z_d$, we obtain that $\fI_{(r,d)}$ is
    \[ \frac{1}{(2\pi\bi)^{k-d}} \oint \cdots \oint \frac{\fA_T(0)^d}{z_{d+1}} \frac{1}{(z_{d+2} - \frac{q}{t} z_{d+1}) \cdots (z_k - \frac{q}{t} z_{k-1})} \prod_{d < i < j \le k} \frac{(z_j - z_i)(z_j - \frac{q}{t}z_i)}{(z_j - qz_i)(z_j - \frac{1}{t} z_i)} \prod_{d < i \leq k} \fA_T(z_i) dz_i. \]
    We may change the contours so that
    \[\fI_{(r,d)} = \frac{1}{(2\pi\bi)^{k-d}} \oint \cdots \oint \frac{\fA_T(0)^d}{w_1} \fB(w_1,\ldots,w_{k-d}) \prod_{i=1}^{k-d} \fA_T(w_i) dw_i \]
    where the $w_i$-contour is $\fY_i$.
    
    \item $0 < r < k-d$. Evaluating the residues at $0$ for $z_{r+1},\ldots,z_{r+d}$, we obtain that $\fI_{(r,d)}$ is
    \[ \frac{-1}{(2\pi\bi)^{k-d}} \oint \cdots \oint \frac{\fA_T(0)^d \frac{q^{r-1}}{t^{r-1}}}{z_r z_{r+d+1}} \prod_{\substack{1 \leq i < k \\ i \notin [r,r+d]}} (z_{i+1} - \frac{q}{t} z_i)^{-1} \prod_{\substack{1 \le i < j \le k \\ i,j \notin (r,r+d]}} \frac{(z_j - z_i)(z_j - \frac{q}{t}z_i)}{(z_j - qz_i)(z_j - \frac{1}{t} z_i)} \prod_{\substack{1 \leq i \leq k \\ i \notin (r,r+d]}} \fA_T(z_i) dz_i. \]
    Reindexing the variables via
    \begin{align*}
    (z_1,\ldots,z_r) & \mapsto (w_1,\ldots,w_r), \\
    (z_{r+d+1},\ldots,z_k) & \mapsto (w_{r+1},\ldots,w_{k-d}),
    \end{align*}
    we obtain
    \[ \fI_{(r,d)} = -  \frac{1}{(2\pi\bi)^{k-d}} \oint \cdots \oint \frac{\fA_T(0)^d \frac{q^{r-1}}{t^{r-1}}}{w_r w_{r+1}} (w_{r+1} - \frac{q}{t} w_r) \fB(w_1,\ldots,w_{k-d}) \prod_{i=1}^{k-d} \fA_T(w_i) dw_i \]
    where we may take the $w_i$-contour to be $\fY_i$.
    
    \item $r = k-d$. Evaluating the residues at $0$ for $z_{r+1},\ldots,z_k$, we obtain that $\fI_{(r,d)}$ is
    \[ \frac{-1}{(2\pi\bi)^{k-d}} \oint \cdots \oint \frac{\fA_T(0)^d \frac{q^{r-1}}{t^{r-1}}}{z_r} \frac{1}{(z_2 - \frac{q}{t}z_1) \cdots (z_r - \frac{q}{t}z_{r-1})} \prod_{1 \le i < j \le r} \frac{(z_j - z_i)(z_j - \frac{q}{t}z_i)}{(z_j - qz_i)(z_j - \frac{1}{t} z_i)} \prod_{1 \le i \le r} \fA_T(z_i) dz_i. \]
    This is exactly
    \[\fI_{(r,d)} = - \frac{1}{(2\pi\bi)^{k-d}} \oint \cdots \oint \frac{\fA_T(0)^d \frac{q^{k-d-1}}{t^{k-d-1}}}{w_{k-d}} \fB(w_1,\ldots,w_{k-d}) \prod_{i=1}^{k-d} \fA_T(w_i) dw_i \]
    where the $w_i$-contour is $\fY_i$.
\end{enumerate}

Fix $0 < d < k$. Combining the three cases and letting $W = (w_1,\ldots,w_{k-d}) $, we may write
\begin{align*}
\sum_{r=0}^{k-d} \fI_{(r,d)} & = \frac{\fA_T(0)^d}{(2\pi \bi)^{k-d}} \oint \cdots \oint dW \cdot \fB(W) \fA_T(W) \left( \frac{1}{w_1} - \sum_{r=1}^{k-d-1} \frac{ \frac{q^{r-1}}{t^{r-1}}}{w_r w_{r+1}} (w_{r+1} - \frac{q}{t} w_r) - \frac{ \frac{q^{k-d-1}}{t^{k-d-1}}}{w_{k-d}}\right)
\end{align*}
because the $w_i$-contours from each of the different cases were set as $\fY_i$ for $1 \le i \le k-d$. The integral vanishes because the parenthesized term may be rewritten as
\[ \frac{1}{w_1} + \sum_{r=1}^{k-d-1} \left( - \frac{1}{w_r} \frac{q^{r-1}}{t^{r-1}} + \frac{1}{w_{r+1}}\frac{q^r}{t^r} \right) - \frac{1}{w_{k-d}} \frac{q^{k-d-1}}{t^{k-d-1}} = 0. \]
Therefore (\ref{eq:YZexp}) simplifies to
\[ \E \fp_k(\lambda^T) = t^{-kN} + \fI_\emptyset \]
which is exactly (\ref{eq:y}).
\end{proof}
\end{example}

The general proof of Proposition \ref{prop:macobs-} is just a factorization into cases of the example above and an inclusion-exclusion argument.

\begin{proof}[Proof of \Cref{prop:macobs-}]
Define contours $\fY_{i,j}, \fZ_{i,j}$ such that (i) $\fY_{i,j}$ enclosed by $t\fY_{i',j'}, q^{-1} \fY_{i',j'}$ and $\fZ_{i,j}$ is enclosed by $t\fZ_{i',j'}, q^{-1} \fY_{i',j'}$ whenever $(i,j) < (i',j')$ in lexicographical order, (ii) $\fY_{i,j}$ is positively oriented around $t^N$ but does not enclose $0$ or $t^{-\alpha_\tau-M_\tau+1}$ for $1 \le \tau \le T_i$ and $\fZ_{i,j}$ is positively oriented around $0$ but does not enclose $t^N$ or $t^{-\alpha_\tau-M_\tau+1}$ for $1 \le \tau \le T_i$, (iii) $\fY_{m,k_m}$ and $\fZ_{m,k_m}$ are disjoint from one another. Note that if contours $\fY_{i,j}$ ($1 \le j \le k_i, 1 \le i \le m$) satisfying (i) and (ii) exist, then the existence of contours $\fZ_{i,j}$ ($1 \le j \le k_i, 1 \le i \le m$) satisfying (i), (ii), (iii) is guaranteed by choosing these contours close enough to $0$.

Let $\cP(Z_1,\ldots,Z_m)$ be the power set of $\{z_{i,j}\}_{1 \le j \le k_i,1 \le i \le m}$. For each element $\Upsilon \in \cP(Z_1,\ldots,Z_m)$, let $\Upsilon_{i,j} = \fZ_{i,j}$ if $z_{i,j} \in \Upsilon$ and $\Upsilon_{i,j} = \fY_{i,j}$ if $z_{i,j} \notin \Upsilon$. Let
\[ \fI_\Upsilon = \frac{1}{(2\pi\bi)^{k_1 + \cdots + k_m}} \oint_{\Upsilon_{1,1}} \cdots \oint_{\Upsilon_{m,k_m}} \prod_{1 \le i < i' \le m} \fC(Z_i,Z_{i'}) \prod_{i=1}^m \fB(Z_i) \fA_{T_i}(Z_i) dZ_i. \]
As in Example \ref{ex:m=1}, we can use Proposition \ref{prop:macobs} and expand the contours as either $\fY_{i,j}$ or $\fZ_{i,j}$ so that
\[ \E\left[ \fp_{k_1}(\lambda^{T_1}) \cdots \fp_{k_m}(\lambda^{T_m}) \right] = \sum_{\Upsilon \in \cP(Z_1,\ldots,Z_m)} \fI_\Upsilon. \]
For an analogous reason as in Example \ref{ex:m=1}, $\fI_\Upsilon = 0$ unless $\Upsilon \cap Z_i$ has the form $\{z_{i,r_i + 1},\ldots,z_{i,r_i + d_i}\}$ for each $1 \le i \le m$ in which case we identify $\Upsilon$ with the ordered pair $(\vec{r},\vec{d})$ where $\vec{r} = (r_1,\ldots,r_m)$ and $\vec{d} = (d_1,\ldots,d_m)$; if $\Upsilon \cap Z_i = \emptyset$ then $d_i = 0$ and take $r_i = 0$ as convention. Let $\cR_{\vec{d}}$ denote the set of $\vec{r}$ where $0 \le r_i \le k_i - d_i$ if $d_i > 0$ and $r_i = 0$ if $d_i = 0$. Then
\begin{align} \label{eq:(r,d)sum}
\E\left[ \fp_{k_1}(\lambda^{T_1}) \cdots \fp_{k_m}(\lambda^{T_m}) \right] = \sum_{\substack{0 \le d_i \le k_i \\ 1 \le i \le m}} \sum_{\vec{r} \in \cR_{\vec{d}}} \fI_{(\vec{r},\vec{d})}.
\end{align}

Fix $1 \le i \le k$. Let $W = (w_1,\ldots,w_k)$, define $A^{(i)}_{0,0}(W) = 1$ and $A^{(i)}_{0,k}(W) = \fA_{T_i}(0)^k = t^{-kN}$. For $0 < d < k$ and $0 \le r \le k-d$ define
\[ A^{(i)}_{r,d}(W) = \left\{ \begin{array}{ll}
\displaystyle \fA_{T_i}(0)^d \frac{1}{w_1} & r = 0, \\
\displaystyle - \fA_{T_i}(0)^d \frac{\frac{q^{r-1}}{t^{r-1}}}{w_r w_{r+1}} (w_{r+1} - \frac{q}{t}w_r) & 0 < r < k - d, \\
\displaystyle - \fA_{T_i}(0)^d \frac{\frac{q^{k-d-1}}{t^{k-d-1}}}{w_{k-d}} & r = k-d.
\end{array} \right. \]
These were the differing parts of the integrand in Example \ref{ex:m=1} from the three cases where we showed
\[ \sum_{r=0}^{k-d} A^{(i)}_{r,d}(W) = 0. \]
Consider $\fI_{(\vec{r},\vec{d})}$ and evaluate the residues of the $\fZ_{i,j}$ contours (which are necessarily at $0$) in lexicographical order of $(i,j)$. Then by similar reasoning as in Example \ref{ex:m=1}, we have
\[ \fI_{(\vec{r},\vec{d})} = \frac{1}{(2\pi\bi)^{\sum_{i=1}^m (k_i - d_i)}} \oint \cdots \oint \prod_{1 \le i < i' \le k} \fC(W_i,W_{i'}) \prod_{i=1}^m \fB(W_i) \fA_{T_i}(W_i) A^{(i)}_{r_i,d_i}(W_i) dW_i  \]
where $W_i = (w_{i,1},\ldots,w_{i,k_i - d_i})$ and the $w_{i,j}$-contour can be taken to be $\fY_{i,j}$. Notice that for any fixed $\vec{d}$, we have
\[ \sum_{\vec{r} \in \cR_{\vec{d}}} \prod_{i=1}^m A^{(i)}_{r_i,d_i}(W_i) = \prod_{i:d_i = 0} A^{(i)}_{0,0}(W_i) \prod_{i: d_i > 0} \sum_{r_i = 0}^{k_i - d_i} A^{(i)}_{r_i,d_i}(W_i) = \prod_{i: d_i > 0} \sum_{r_i = 0}^{k_i - d_i} A^{(i)}_{r_i,d_i}(W_i)  \]
where the latter is zero unless we have that $d_i = 0$ or $d_i = k_i$ for each $1 \le i \le m$. Thus (\ref{eq:(r,d)sum}) becomes
\begin{align} \label{eq:Isum}
\E\left[ \fp_{k_1}(\lambda^{T_1}) \cdots \fp_{k_m}(\lambda^{T_m}) \right] = \sum_{\substack{\Upsilon \cap Z_i = \emptyset ~\mbox{\footnotesize or}~ Z_i \\ 1 \le i \le m}} \fI_\Upsilon.
\end{align}

Given a subset $S \subset [[1,m]]$, let
\[ \cI_S = \frac{1}{(2\pi \bi)^{\sum_{i \in S} k_i}} \oint \cdots \oint \prod_{\substack{1 \le i < i' \le k \\ i, i' \in S}} \fC(W_i,W_{i'}) \prod_{i \in S} \fB(W_i) \fA_{T_i}(W_i) dW_i \]
where the $w_{i,j}$-contour is $\fY_{i,j}$. Once again, by evaluating the residues of the $\fZ_{i,j}$ contours in lexicographical order of $(i,j)$, we have
\[ \fI_{\bigcup_{i \notin S} Z_i} = \cI_S \prod_{i \notin S} \fA_{T_i}(0)^{k_i} = t^{-N \sum_{i \in [[1,m]] \setminus S} k_i} \cdot \cI_S. \]
Since (\ref{eq:Isum}) is a sum over those $\Upsilon$ which are unions of $Z_i$, we may write
\[ \E\left[ \fp_{k_1}(\lambda^{T_1}) \cdots \fp_{k_m}(\lambda^{T_m}) \right] = \sum_{S \subset [[1,m]]} t^{-N \sum_{i \in [[1,m]] \setminus S} k_i} \cdot \cI_S. \]
In particular, observe that for any $S \subset [[1,m]]$,
\[ \E\left[ \prod_{i \in S} \fp_{k_i}(\lambda^{T_i}) \right] = \sum_{S' \subset S} t^{-N\sum_{i \in S \setminus S'} k_i} \cdot \cI_{S'}. \]
Then by inclusion-exclusion
\[ \cI_{[[1,m]]} = \sum_{S \subset [[1,m]]} (-1)^{m-|S|} t^{-N \sum_{i \in [[1,m]] \setminus S} k_i} \E\left[ \prod_{i \in S} \fp_{k_i}(\lambda^{T_i}) \right] = \E\left[ \prod_{i=1}^m \left(\fp_{k_i}(\lambda^{T_i}) - t^{-k_iN} \right) \right] \]
which completes the proof.
\end{proof}

The proof of \Cref{prop:schurobs-} follows in the same manner as that of \Cref{prop:macobs-}; we omit the proof to avoid repetition.

\subsection{Jacobi Limit}
Using the observables from Proposition \ref{prop:macobs-} and the Jacobi limit in Theorem \ref{thm:mactoprod}, we obtain observables for the $\beta$-Jacobi product process.

\begin{proof}[Proof of \Cref{thm:obs}]
For $\e > 0$, let $q = e^{-\e}, t = q^\theta$ and $(\lambda^T := \lambda^T(\e))_{T \in \Z_+} \sim \P_{q,t}^{\bA,N}$. Observe that for $k \ge 0$ we have
\[ t^{-kN} - \fp_k(\lambda^T;q,t) = (e^{\theta k \e} - 1) \sum_{i=1}^N e^{-k\e\lambda_i^T} e^{-\theta k\e (-i+1)}. \]
By the convergences
\[ \frac{e^{\theta k\e} - 1}{\theta k\e} \to 1, \quad \quad \quad \quad e^{\theta k\e(-i+1)} \to 1 \quad \quad (1 \le i \le N) \]
as $\e \to 0$ and \Cref{thm:mactoprod}, we have the convergence
\begin{align} \label{eq:pk_conv}
\lim_{\e \to 0} \E\left[  \frac{t^{-k_1N} - \fp_{k_1}(\lambda^{T_1};q,t)}{k_1\theta\e} \cdots \frac{t^{-k_mN} - \fp_{k_m}(\lambda^{T_m};q,t)}{k_m\theta\e} \right] \to \E\left[ \fP_{k_1}(\by^{(T_1)}) \cdots \fP_{k_m}(\by^{(T_m)}) \right]
\end{align}
as $\e \to 0$, where $(\by^{(T)})_{T \in \Z_+} \sim \P_\theta^{\bA,N}$.

We compute the limit on the left hand side of (\ref{eq:pk_conv}) to obtain an expression for the right hand side. Proposition \ref{prop:macobs-} yields the contour integral formula
\[ \E \left[ \prod_{i=1}^m \frac{t^{-k_iN} - \fp_{k_i}(\lambda^{T_i};q,t)}{k_i\theta\e} \right] = \frac{\prod_{i=1}^m (-k_i\theta \e)^{-1}}{(2\pi\bi)^{\sum_{i=1}^m k_i}} \oint \cdots \oint \prod_{1 \le i < i' \le m} \fC(Z_i,Z_{i'}) \prod_{i=1}^m \fB(Z_i) \fA_{T_i}(Z_i) \, dZ_i \]
where $Z_i = (z_{i,1},\ldots,z_{i,k_i})$,
\begin{itemize}
    \item the $z_{i,j}$-contour $\fY_{i,j}$ is positively oriented around $t^N$ but does not enclose $0$ or $t^{-\alpha_\tau-M_\tau +1}$ for $1 \le \tau \le T_i$;
    \item the $\fY_{i,j}$ contour is enclosed by $t \fY_{i',j'}, q^{-1} \fY_{i',j'}$ for $(i,j) < (i',j')$ in lexicographical order.
\end{itemize}

Note that we may take the $z_{i,j}$-contours close to $t^N$. Changing variables $z_{i,j} = e^{\e u_{i,j}}$, we have
\begin{align*}
\begin{multlined}
\E \left[ \prod_{i=1}^m \frac{t^{-k_iN} - \fp_{k_i}(\lambda^{T_i};q,t)}{k_i\theta\e} \right] \\
= \frac{\prod_{i=1}^m (-k_i\theta)^{-1}}{(2\pi\bi)^{\sum_{i=1}^m k_i}} \oint \cdots \oint \prod_{1 \le i < i' \le m} \fC(e^{\e U_i},e^{\e U_{i'}}) \prod_{i=1}^m \e^{k_i-1} \fB(e^{\e U_i}) \fA_{T_i}(e^{\e U_i}) \, e^{\e \sum_{j=1}^{k_i} u_{i,j}} \,dU_i
\end{multlined}
\end{align*}
where $U_i = (u_{i,1},\ldots,u_{i,k_i})$, we denote $e^{\e U_i} := (e^{\e u_{i,1}},\ldots,e^{\e u_{i,k_i}})$,
\begin{itemize}
    \item the $u_{i,j}$-contour $\fU_{i,j}$ is positively oriented around the pole at $-\theta N$ but does not enclose $\theta(\alpha_\tau + M_\tau - 1)$ for $1 \le \tau \le T_i$;
    \item whenever $(i,j) < (i',j')$ in lexicographical order, $\fU_{i,j}$ is enclosed by $\fU_{i',j'} - \theta, \fU_{i',j'} + 1$.
\end{itemize}
These contours are independent of $\e > 0$. We have the convergences
\[ \fA_T(e^{\e U_i}) \to \cA_T(U_i), \quad \quad \e^{k_i - 1} \fB(e^{\e U_i}) \to k_i \cB(U_i), \quad \quad \fC(e^{\e U_i, \e U_{i'}}) \to \cC(U_i,U_{i'}) \]
as $\e \to 0$ uniformly over $u_{i,j} \in \fU_{i,j}$ where $1 \le j \le k_i, 1 \le i \le m$. Therefore
\[ \lim_{\e \to 0} \E \left[ \prod_{i=1}^m \frac{t^{-k_iN} - \fp_{k_i}(\lambda^{T_i};q,t)}{k_i\theta\e} \right] = \frac{(-\theta)^{-m}}{(2\pi\bi)^{\sum_{i=1}^m k_i}} \oint \cdots \oint \prod_{1 \le i < i' \le m} \cC(U_i,U_{i'}) \prod_{i=1}^m \cB(U_i) \cA_{T_i}(U_i) \, dU_i \]
where the $u_{i,j}$-contour is $\fU_{i,j}$ as defined above. The theorem now follows from (\ref{eq:pk_conv}).
\end{proof}

The proof of \Cref{thm:obs_cpx} is similar. We highlight some of the modifications, but omit details to avoid repetition.

\begin{proof}[Proof of \Cref{thm:obs_cpx}]
For $\e > 0$, let $q = e^{-\e}$, $t_i = q^{c_i}$ for $1 \le i \le m$ and $(\lambda^T := \lambda^T(\e))_{T \in \Z_+} \sim \P_{q,q}^{\bA,N}$. \Cref{thm:obs_cpx} now follows by
\[ \lim_{\e \to 0} \E \left[ \prod_{i=1}^m \frac{t_i^{-k_iN} - \boldsymbol{p}_{t_i}(\lambda^{T_i})}{c_i \e} \right] = \E \left[ \prod_{i=1}^m \fP_{c_i}(\by^{(T_i)}) \right] \]
and the change of variables $z_i = e^{\e u_i}$ on the formula given by \Cref{prop:schurobs-}.
\end{proof}

\section{Global Asymptotics} \label{sec:global}

In this section we prove \Cref{thm:lln,thm:clt}. We begin by reformulating the theorems in terms of convergence of moments.

\begin{assumption} \label{assum:global}
Fix $\theta > 0$. Let $\bA := (\alpha_T,M_T)_{T \in \Z_+} \in (\R_+ \times \Z_+)^{\Z_+}$ vary with $N \in \Z_+$ such that
\[ \frac{\alpha_T}{N} \to \wh{\alpha}_T \quad \mbox{and} \quad \frac{M_T}{N} \to \wh{M}_T \]
as $N \to \infty$ for some $(\wh{\alpha}_T, \wh{M}_T) \in (\R_{\ge 0}^2)^{\Z_+}$. Let $(\by^{(T)} := \by^{(T)}(N))_{T \in \Z_+} \sim \P_\theta^{\bA,N}$. 
\end{assumption}

\begin{theorem} \label{prop:lln}
Under Assumption \ref{assum:global}, for any $k,T \in \Z_+$ we have
\[ \lim_{N\to\infty} \frac{1}{N} \E[ \fP_k(\by^{(T)}) ] = - \frac{1}{k} \cdot \frac{1}{2\pi\bi} \oint \left( \frac{v}{v+1} \prod_{\tau = 1}^T \frac{v - \wh{\alpha}_\tau }{v - \wh{\alpha}_\tau - \wh{M}_\tau} \right)^k \, dv \]
where the contour is positively oriented around the pole at $-1$ but does not enclose $\wh{\alpha}_\tau + \wh{M}_\tau$ for $1 \le \tau \le T$.
\end{theorem}

\begin{theorem} \label{prop:covariance}
Under Assumption \ref{assum:global}, for any positive integers $k_1,k_2,T_1 \ge T_2$ we have
\[ \lim_{N\to\infty} \cov\left(\fP_{k_1}(\by^{(T_1)}),\fP_{k_2}(\by^{(T_2)}) \right) = \frac{\theta^{-1}}{(2\pi\bi)^2} \oint \oint  \frac{1}{(v_2 - v_1)^2} \prod_{i=1}^2 \left( \frac{v_i}{v_i + 1} \prod_{\tau=1}^{T_i} \frac{v_i - \wh{\alpha}_\tau}{v_i - \wh{\alpha}_\tau - \wh{M}_\tau} \right)^{k_i} \, dv_i \]
where the $v_2$-contour encloses the $v_1$-contour, and for each $i = 1,2$, the $v_i$-contour is positively oriented around $-1$, but does not enclose $\wh{\alpha}_\tau + \wh{M}_\tau$ for $1 \le \tau \le T_i$.
\end{theorem}

\begin{theorem} \label{prop:high_cumulant}
Under Assumption \ref{assum:global}, for any positive integers $m \ge 3, k_1,\ldots,k_m,T_1,\ldots,T_m$ we have
\[ \kappa\left( \fP_{k_1}(\by^{(T_1)}), \cdots, \fP_{k_m}(\by^{(T_m)}) \right) \to 0 \]
as $N\to\infty$ where $\kappa$ denotes the joint cumulant (see \Cref{def:cumulant1}).
\end{theorem}

We show how these theorems imply \Cref{thm:lln,thm:clt}.

\begin{proof}[Proof of \Cref{thm:lln}]
This follows from the convergence in \Cref{prop:lln} and the fact that
\[ \lim_{N\to\infty} \mathrm{Var}\left( \frac{1}{N} \fP_k(\by^{(T)}) \right) = 0 \]
by \Cref{prop:covariance}.
\end{proof}

\begin{proof}[Proof of \Cref{thm:clt}]
\Cref{prop:covariance} implies the covariance formula \eqref{eq:asymp_cov}. To prove asymptotic Gaussianity, first observe that for $\nu \ge 3$ and $1 \le i_1,\ldots,i_\nu \le k$, \eqref{eq:cumulant2} gives
\[ \kappa\Bigl(\fP_{k_{i_1}}(\by^{(T_{i_1})}) - \E[\fP_{k_{i_1}}(\by^{(T_{i_1})})], \ldots, \fP_{k_{i_\nu}}(\by^{(T_{i_\nu})}) - \E[\fP_{k_{i_\nu}}(\by^{(T_{i_\nu})})] \Bigr) = \kappa\Bigl(\fP_{k_{i_1}}(\by^{(T_{i_1})}), \ldots, \fP_{k_{i_\nu}}(\by^{(T_{i_\nu})}) \Bigr). \]
By \Cref{prop:high_cumulant}, the latter converges to $0$. \Cref{lem:gauss_cumulant} then implies asymptotic Gaussianity of \eqref{eq:clt_vector}.
\end{proof}

The remainder of this section is devoted to the proofs of \Cref{prop:lln,prop:covariance,prop:high_cumulant}. We state a useful lemma.

\begin{lemma}[{\cite{GZ}*{Corollary A.2}}] \label{lem:dim_red}
Let $s \in \Z_+$ and $f,g_1\ldots,g_s$ be meromorphic functions with possible poles at $\{P_1,\ldots,P_m\}$. Then for $n \ge 2$,
\[ \frac{1}{(2\pi\bi)^n} \oint \cdots \oint \frac{1}{(v_2 - v_1) \cdots (v_n - v_{n-1})} \left( \prod_{i=1}^s \sum_{j=1}^n g_i(v_j) \right) \prod_{i=1}^n f(v_i) dv_i = \frac{n^{s-1}}{2\pi\bi} \oint f(v)^n \prod_{i=1}^s g_i(v) dv, \]
where the contours on both sides are positively oriented around $\{P_1,\ldots,P_m\}$, and for the left hand side we require the $v_j$-contour to enclose the $v_i$-contour for $1 \le i < j \le n$.
\end{lemma}

\subsection*{Proof of \texorpdfstring{\Cref{prop:lln}}{Law of Large Numbers}}

By Theorem \ref{thm:obs}, we have
\begin{align*}
\begin{multlined}
\E \left[ \fP_k(\by^{(T)}) \right] = \frac{-\theta^{-1}}{(2\pi\bi)^k} \oint \cdots \oint \frac{1}{(u_2 - u_1 + 1 - \theta) \cdots (u_k - u_{k-1} + 1 - \theta)} \\
\times \prod_{1 \le i < j \le k} \frac{(u_j - u_i)(u_j - u_i + 1 - \theta)}{(u_j - u_i + 1)(u_j - u_i - \theta)} \cdot \prod_{i=1}^k \frac{u_i}{u_i + \theta N} \prod_{\tau = 1}^T \frac{u_i - \theta (\alpha_\tau - 1)}{u_i - \theta(\alpha_\tau + M_\tau - 1)} \, du_i.
\end{multlined}
\end{align*}
where
\begin{enumerate}
    \item the $u_i$-contour $\fU_i$ is positively oriented around $-\theta N$ but does not enclose $\theta(\alpha_\tau + M_\tau - 1)$ for $1 \le \tau \le T$;
    \item whenever $i < j$, $\fU_i$ is enclosed by $\fU_j - \theta, \fU_j + 1$.
\end{enumerate}

Changing variables $u_i = v_i \theta N$, we obtain
\begin{align*}
\begin{multlined}
\frac{1}{N} \E \left[ \fP_k(\by^{(T)}) \right] = \frac{-1}{(2\pi\bi)^k} \oint \cdots \oint \frac{1}{(v_2 - v_1 + \frac{1 - \theta}{\theta N}) \cdots (v_k - v_{k-1} + \frac{1 - \theta}{\theta N})} \\
\times \prod_{1 \le i < j \le k} \frac{(v_j - v_i)(v_j - v_i + \frac{1 - \theta}{\theta N})}{(v_j - v_i + \frac{1}{\theta N})(v_j - v_i - \frac{1}{N})} \cdot \prod_{i=1}^k \frac{v_i}{v_i + 1} \prod_{\tau = 1}^T \frac{v_i - \frac{\alpha_\tau - 1}{N}}{v_i - \frac{\alpha_\tau + M_\tau - 1}{N}} \, dv_i.
\end{multlined}
\end{align*}
Then
\begin{align*}
\begin{multlined}
\lim_{N\to\infty} \frac{1}{N} \E \left[ \fP_k(\by^{(T)}) \right] = \frac{-1}{(2\pi\bi)^k} \oint \cdots \oint \frac{1}{(v_2 - v_1) \cdots (v_k - v_{k-1})} \prod_{i=1}^k \frac{v_i}{v_i + 1} \prod_{\tau = 1}^T \frac{v_i - \wh{\alpha}_\tau }{v_i - \wh{\alpha}_\tau - \wh{M}_\tau} \, dv_i.
\end{multlined}
\end{align*}
where
\begin{enumerate}
    \item the $v_i$-contour $\fV_i$ is positively oriented around $-1$ but does not enclose $\wh{\alpha}_\tau + \wh{M}_\tau$ for $1 \le \tau \le T$;
    \item whenever $i < j$, $\fV_i$ is enclosed by $\fV_j$.
\end{enumerate}
The theorem now follows by applying \Cref{lem:dim_red} to the multidimensional integral above. \qed

\subsection*{Proof of \texorpdfstring{\Cref{prop:covariance}}{Covariance}}

By Theorem \ref{thm:obs}, we have
\begin{align*}
\begin{multlined}
\cov\left(\fP_{k_1}(\by^{(T_1)}),\fP_{k_2}(\by^{(T_2)}) \right) = \frac{\theta^{-2}}{(2\pi\bi)^{k_1 + k_2}} \oint \cdots \oint \left( \cC(U_1,U_2) - 1 \right) \prod_{i=1}^2 \cB(U_i) \cA_{T_i}(U_i) \,dU_i
\end{multlined}
\end{align*}
where $U_i = (u_{i,1},\ldots,u_{i,k_i})$ for $i = 1,2$, 
\begin{enumerate}
    \item the $u_{i,j}$-contour $\fU_{i,j}$ is positively oriented around $-\theta N$ but does not enclose $\theta(\alpha_\tau + M_\tau - 1)$ for $1 \le \tau \le T_i$;
    \item whenever $(i,j) < (i',j')$ in lexicographical order, $\fU_{i,j}$ is enclosed by $\fU_{i',j'} - \theta, \fU_{i',j'} + 1$.
\end{enumerate}

Changing variables $u_{i,j} = v_{i,j} \theta N$, we obtain
\begin{align*}
\begin{multlined}
\cov\left(\fP_{k_1}(\by^{(T_1)}),\fP_{k_2}(\by^{(T_2)}) \right) = \frac{1}{(2\pi\bi)^{k_1 + k_2}} \oint \cdots \oint N^2 \left( \prod_{i=1}^{k_1} \prod_{j=1}^{k_2} \frac{(v_{2,j} - v_{1,i})(v_{2,j} - v_{1,i} + \frac{1 - \theta}{\theta N})}{(v_{2,j} - v_{1,i} + \frac{1}{\theta N})(v_{2,j} - v_{1,i} - \frac{1}{N})} - 1 \right) \\
\times \prod_{i=1}^2 \frac{1}{(v_{i,2} - v_{i,1} + \frac{1 - \theta}{\theta N}) \cdots (v_{i,k_i} - v_{i,k_i-1} + \frac{1 - \theta}{\theta N})} \left( \prod_{1 \le j < j' \le k_i} \frac{(v_{i,j'} - v_{i,j})(v_{i,j'} - v_{i,j} + \frac{1 - \theta}{\theta N})}{(v_{i,j'} - v_{i,j} + \frac{1}{\theta N})(v_{i,j'} - v_{i,j} - \frac{1}{N})} \right) \\
\times \prod_{j=1}^{k_i} \frac{v_{i,j}}{v_{i,j} + 1} \prod_{\tau = 1}^{T_i} \frac{v_{i,j} - \frac{\alpha_\tau - 1}{N}}{v_{i,j} - \frac{\alpha_\tau + M_\tau - 1}{N}} \, dv_{i,j}.
\end{multlined}
\end{align*}
We may take the $v_{i,j}$-contours to be independent of $N$ for $N$ sufficiently large. Observe that
\begin{align} \label{eq:cross_term_asymp}
\begin{split}
\prod_{i=1}^{k_1} \prod_{j=1}^{k_2} \frac{(v_{2,j} - v_{1,i})(v_{2,j} - v_{1,i} + \frac{1 - \theta}{\theta N})}{(v_{2,j} - v_{1,i} + \frac{1}{\theta N})(v_{2,j} - v_{1,i} - \frac{1}{N})} &= \prod_{i=1}^{k_1} \prod_{j=1}^{k_2} \left( 1 + \frac{\theta^{-1}}{N^2}\cdot\frac{1}{(v_{2,j} - v_{1,i} + \frac{1}{\theta N})(v_{2,j} - v_{1,i} - \frac{1}{N})} \right) \\
&= 1 + \frac{\theta^{-1}}{N^2} \sum_{i=1}^{k_1} \sum_{j=1}^{k_2} \frac{1}{(v_{2,j} - v_{1,i})^2} + O\left( \frac{1}{N^3} \right)
\end{split}
\end{align}
where the $O(1/N^3)$ term is uniform over the $v_{i,j}$-contours. Then
\begin{align*}
\begin{multlined}
\lim_{N\to\infty} \cov\left(\fP_{k_1}(\by^{(T_1)}),\fP_{k_2}(\by^{(T_2)}) \right) = \frac{\theta^{-1}}{(2\pi\bi)^{k_1 + k_2}} \oint \cdots \oint \sum_{i=1}^{k_1} \sum_{j=1}^{k_2} \frac{1}{(v_{2,j} - v_{1,i})^2} \\
\times \prod_{i=1}^2 \frac{1}{(v_{i,2} - v_{i,1}) \cdots (v_{i,k_i} - v_{i,k_i-1})} \prod_{j=1}^{k_i} \frac{v_{i,j}}{v_{i,j} + 1} \prod_{\tau = 1}^{T_i} \frac{v_{i,j} - \wh{\alpha}_\tau}{v_{i,j} - \wh{\alpha}_\tau - \wh{M}_\tau} \, dv_{i,j}.
\end{multlined}
\end{align*}
where
\begin{enumerate}
    \item the $v_{i,j}$-contour $\fV_{i,j}$ is positively oriented around $-1$ and does not enclose $\wh{\alpha}_\tau + \wh{M}_\tau$ for $1 \le \tau \le T_i$;
    \item whenever $(i,j) < (i',j')$, $\fV_{i,j}$ is enclosed by $\fV_{i',j'}$.
\end{enumerate}
The theorem now follows by applying \Cref{lem:dim_red} twice to the multidimensional integral above. \qed

\subsection*{Proof of \texorpdfstring{\Cref{prop:high_cumulant}}{High Cumulant}}

By Definition \ref{def:cumulant1}, we have
\[ \kappa\left( \fP_{k_1}(\by^{(T_1)}), \cdots, \fP_{k_m}(\by^{(T_m)}) \right) = \sum_{\substack{d > 0 \\ \{S_1,\ldots,S_d\} \in \Theta_m}} (-1)^{d-1} (d-1)! \prod_{\ell=1}^d \E \left[ \prod_{i \in S_\ell} \fP_{k_i}(\by^{(T_i)}) \right] \]
where $\Theta_m$ denotes the collection of all set partitions of $[[1,m]]$. By symmetry of the cumulant, we may assume $T_1 \ge \cdots \ge T_m$ so that the conditions of \Cref{thm:obs} are met. Then
\begin{align} \label{eq:higher_cumulant}
\kappa\left( \fP_{k_1}(\by^{(T_1)}), \cdots, \fP_{k_m}(\by^{(T_m)}) \right) = \frac{(-\theta)^{-m}}{(2\pi\bi)^{\sum_i k_i}} \oint \cdots \oint \wt{\cC}(U_1,\ldots,U_m) \prod_{i=1}^m \cB(U_i) \cA_{T_i}(U_i) \, dU_i
\end{align}
where $U_i = (u_{i,1},\ldots,u_{i,k_i})$ and
\begin{align} \label{eq:cross_term}
\wt{\cC}(U_1,\ldots,U_m) = \sum_{\substack{d > 0 \\ \{S_1,\ldots,S_d\} \in \Theta_m}} (-1)^{d-1} (d-1)! \prod_{\ell=1}^d \prod_{\substack{(i,j) \in S_\ell \\ i < j}} \cC(U_i,U_j).
\end{align}
We note that the contours are such that
\begin{enumerate}
    \item the $u_{i,j}$-contour $\fU_{i,j}$ is positively oriented around $-\theta N$ but does not enclose $\theta(\alpha_\tau + M_\tau - 1)$ for $1 \le \tau \le T_i$;
    \item whenever $(i,j) < (i',j')$, $\fU_{i,j}$ is enclosed by $\fU_{i',j'} - \theta, \fU_{i',j'} + 1$.
\end{enumerate}

Let $S \subset [[1,m]]$, $\cT(S)$ denote the set of undirected simple graphs with vertices labeled by $S$ and $\cL(S) \subset \cT(S)$ denote the subset of connected graphs. Given a graph $\Omega$, we denote by $E(\Omega)$ the edge set of $\Omega$. We claim that
\begin{align} \label{eq:mobius}
\wt{\cC}(U_1,\ldots,U_m) = \sum_{\Omega \in \cL([[1,m]])} \prod_{\substack{(i,j) \in E(\Omega) \\ i < j}} ( \cC(U_i,U_j) - 1).
\end{align}
Define
\[ \cK(S) := \sum_{\Omega \in \cL(S)} \prod_{\substack{(i,j) \in E(\Omega) \\ i < j}} (\cC(U_i,U_j) - 1), \quad \cE(S) := \sum_{\Omega \in \cT(S)} \prod_{\substack{(i,j) \in E(\Omega) \\ i < j}} (\cC(U_i,U_j) - 1). \]
Then
\[ \cE(S) = \sum_{\substack{d > 0 \\ \{S_1,\ldots,S_d\} \in \Theta_S}} \prod_{\ell=1}^d \cK(S_\ell) \]
where $\Theta_S$ is the collection of set partitions of $S$. By \Cref{lem:cumulant_inverse}, we have
\[ \cK(S) = \sum_{\substack{d > 0 \\ \{S_1,\ldots,S_d\} \in \Theta_S}} (-1)^{d-1} (d-1)! \prod_{\ell=1}^d \cE(S_\ell) \]
which agrees with the right hand side of \eqref{eq:cross_term} when $S = [[1,m]]$. Thus \eqref{eq:mobius} follows.

Applying \eqref{eq:mobius}, \eqref{eq:higher_cumulant} becomes
\begin{align*}
\kappa\left( \fP_{k_1}(\by^{(T_1)}) \cdots \fP_{k_m}(\by^{(T_m)}) \right) &= \sum_{\Omega \in \cL([[1,m]])} \frac{(-\theta)^{-m}}{(2\pi\bi)^{\sum_i k_i}} \oint \cdots \oint \! \prod_{\substack{(i,j) \in E(\Omega) \\ i < j}} (\cC(U_i,U_j) - 1) \prod_{i=1}^m \cB(U_i) \cA_{T_i}(U_i) \, dU_i \\
&=: \sum_{\Omega \in \cL([[1,m]])} \cI_\Omega.
\end{align*}

We want to show that $\cI_\Omega \sim o(1)$ for each $\Omega \in \cL([[1,m]])$ as $N\to\infty$. Fix $\Omega \in \cL([[1,m]])$. Changing variables $u_{i,j} = \theta N v_{i,j}$, we obtain
\begin{align}
\begin{multlined}
\cI_\Omega = \frac{(-N)^m}{(2\pi\bi)^{\sum_i k_i}} \oint \cdots \oint \prod_{\substack{(i,j) \in E(\Omega) \\ i < j}} \left( \prod_{a=1}^{k_i} \prod_{b=1}^{k_j} \frac{(v_{j,b} - v_{i,a})(v_{j,b} - v_{i,a} + \frac{1 - \theta}{\theta N})}{(v_{j,b} - v_{i,a} + \frac{1}{\theta N})(v_{j,b} - v_{i,a} - \frac{1}{N})} - 1 \right) \\
\times \prod_{i=1}^m \frac{1}{(v_{i,2} - v_{i,1} + \frac{1 - \theta}{\theta N}) \cdots (v_{i,k_i} - v_{i,k_i - 1} + \frac{1 - \theta}{\theta N})} \left( \prod_{1 \le a < b \le k_i} \frac{(v_{i,b} - v_{i,a})(v_{i,b} - v_{i,a} + \frac{1-\theta}{\theta N})}{(v_{i,b} - v_{i,a} + \frac{1}{\theta N})(v_{i,b} - v_{i,a} - \frac{1}{N})} \right) \\
\times \prod_{j=1}^{k_i} \frac{v_{i,j}}{v_{i,j} + 1} \prod_{\tau=1}^{T_i} \frac{v_{i,j} - \frac{\alpha_\tau - 1}{N}}{v_{i,j} - \frac{\alpha_\tau + M_\tau - 1}{N}} dv_{i,j}.
\end{multlined}
\end{align}
The contours may be chosen to be fixed for sufficiently large $N$. In the same manner that we have \eqref{eq:cross_term_asymp} in the proof of \Cref{prop:covariance}, we have
\begin{align*}
\prod_{a=1}^{k_i} \prod_{b=1}^{k_j} \frac{(v_{j,b} - v_{i,a})(v_{j,b} - v_{i,a} + \frac{1 - \theta}{\theta N})}{(v_{j,b} - v_{i,a} + \frac{1}{\theta N})(v_{j,b} - v_{i,a} - \frac{1}{N})} &= 1 + \frac{\theta^{-1}}{N^2} \sum_{a=1}^{k_i} \sum_{b=1}^{k_j} \frac{1}{(v_{j,b} - v_{i,a})^2} + O\left(\frac{1}{N^3}\right)
\end{align*}
where $O(1/N^3)$ is uniform over our contours. This implies that $\cI_\Omega \sim O(N^{m - 2|E(\Omega)|})$. For any $\Omega \in \cL([[1,m]])$, we have that $|E(\Omega)| \ge m - 1$. Therefore $\cI_\Omega \sim o(1)$ whenever $m \ge 3$. This completes the proof. \qed

\section{Local Fluctuations at the Edge} \label{sec:local}

In this section, we prove the main results \Cref{thm:local_cpx_intro,thm:local_cpx_jacobi_intro,thm:local_intro} for local fluctuations. We begin by reformulating the main theorems.

\begin{assumption} \label{assum:local}
Fix $\theta > 0$. Let $T_1,\ldots,T_m \in \Z_+$ and $\bA := (\alpha_T,M_T)_{T \in \Z_+} \in (\R_+ \times \Z_+)^{\Z_+}$ vary with $N$ such that
\begin{align*}
\begin{gathered}
\liminf_{N\to\infty} \left( \inf_{T > 0} \alpha_T \right) > 0, \quad \quad \liminf_{N\to\infty} \left( \inf_{T > 0} M_T/N \right) > 0  \\
\lim_{N \to \infty} \sum_{\tau =1}^{T_i} \left( \frac{1}{N + \alpha_\tau - 1} - \frac{1}{N + \alpha_\tau + M_\tau - 1} \right) = \gamma_i
\end{gathered}    
\end{align*}
for some $\gamma_i > 0$, $1 \le i \le m$. Let $(\by^{(T)} := \by^{(T)}(N))_{T \in \Z_+} \sim \P_\theta^{\bA,N}$.
\end{assumption}

\begin{remark}
The assumption $\liminf_N(\inf \alpha_T) > 0$ is purely technical and should be removable with some additional effort. \Cref{thm:local} below can be shown without this assumption.
\end{remark}

\begin{theorem} \label{thm:local_cpx}
Let $0 < c_1,\ldots,c_m < 1/m$. Let $T_1,\ldots,T_m \in \Z_+$ vary with $N$ such that
\[ \frac{T_i}{N} \to \wh{T}_i \]
as $N\to\infty$ for some $\wh{T}_i > 0$, $1 \le i \le m$. If $(\by^{(T)} := \by^{(T)}(N))_{T \in \Z_+}$ is distributed as the squared singular values of $Y_T = X_T \cdots X_1$ where $X_1,X_2,\ldots$ are independent $N\times N$ complex Ginibre matrices (i.e. iid standard complex entries), then
\begin{align*}
\begin{multlined}
\lim_{N\to\infty} \E \left[ \prod_{i=1}^m \frac{1}{N^{c_i(T_i + 1)}} \fP_{c_i}(\by^{(T_i)}) \right] \\
= \frac{1}{(2\pi\bi)^m} \oint \cdots \oint \det \left[ \frac{1}{u_i - u_j - c_j} \right]_{i,j=1}^m \prod_{i=1}^m \frac{\Gamma(-u_i-c_i)}{\Gamma(-u_i)} \exp\left[ -\wh{T}_i \left( \frac{c_i(c_i + 1)}{2} + c_i u_i \right) \right] du_i
\end{multlined}
\end{align*}
where the $u_i$-contour $\fU_i$ is a positively oriented contour around $-c_i,-c_i+1,-c_i+2,\ldots$ which starts and ends at $+\infty$, and $\fU_i$ is enclosed by $\fU_j - c_i, \fU_j + c_j$ for $i < j$.
\end{theorem}

This gives us \Cref{thm:local_cpx_intro}. By combining \cite{LWW}*{Theorem 3.2} with \Cref{thm:local_cpx}, we obtain \Cref{thm:laplace}. The next theorem implies \Cref{thm:local_cpx_jacobi_intro}.

\begin{theorem} \label{thm:local_cpx_jacobi}
Let $0 < c_1,\ldots,c_m < \liminf_{N\to\infty}(\inf_{T > 0} \alpha_T)/m$. Under Assumption \ref{assum:local} with $\theta = 1$, we have
\begin{align*}
\begin{multlined}
\lim_{N\to\infty} \E \left[ \prod_{i=1}^m \left( \frac{1}{N} \prod_{\tau = 1}^{T_i} \frac{N + M_\tau + \alpha_\tau - 1}{N + \alpha_\tau - 1} \right)^{c_i} \fP_{c_i}(\by^{(T_i)}) \right] \\
= \frac{1}{(2\pi\bi)^m} \oint \cdots \oint \det \left[ \frac{1}{u_i - u_j - c_j} \right]_{i,j=1}^m
\prod_{i=1}^m \frac{\Gamma(-u_i-c_i)}{\Gamma(-u_i)} \exp\left[ - \gamma_i \left( \frac{c_i(c_i+1)}{2} + c_i u_i \right) \right] du_i
\end{multlined}
\end{align*}
where the $u_i$-contour $\fU_i$ is a positively oriented contour around $-c_i,-c_i+1,-c_i+2,\ldots$ which starts and ends at $+\infty$, and $\fU_i$ is enclosed by $\fU_j - c_i, \fU_j + c_j$ for $i < j$.
\end{theorem}

Our final theorem for arbitrary $\beta = 2\theta > 0$ implies \Cref{thm:local_intro}.

\begin{theorem} \label{thm:local}
Let $k_1,\ldots,k_m$ be positive integers. Under Assumption \ref{assum:local}, we have
\begin{align*}
& \lim_{N\to\infty} \E\left[ \prod_{i=1}^m \left( \frac{1}{N} \prod_{\tau = 1}^{T_i} \frac{N + M_\tau + \alpha_\tau - 1}{N + \alpha_\tau - 1} \right)^{k_i} \fP_{k_i}(\by^{(T_i)}) \right] \\
&\quad \quad = \frac{\theta^{-m}}{(2\pi\bi)^{\sum_{i=1}^m k_i}} \oint \cdots \oint \prod_{\substack{1 \le i,i' \le m, 1 \le j \le k_i, 1 \le j' \le k_{i'} \\ (i,j) < (i,j')}} \frac{(u_{i',j'} - u_{i,j})(u_{i',j'} - u_{i,j} + 1 - \theta)}{(u_{i',j'} - u_{i,j} + 1)(u_{i',j'} - u_{i,j} - \theta)} \\
& \quad \quad \quad \quad \times \prod_{i=1}^m \left( \prod_{a=1}^{k_i-1} \frac{1}{u_{i,a} - u_{i,a+1} - 1 + \theta} \right) \left( \prod_{j=1}^{k_i} \frac{e^{-\gamma_i u_{i,j} / \theta }}{u_{i,j}/\theta} \, du_{i,j} \right)
\end{align*}
where the $u_{i,j}$-contour is positively oriented around $0$, the $u_{i',j'}$-contour encloses a $\max(\theta,1)$-neighborhood of the $u_{i,j}$-contour for $(i,j) < (i',j')$ in lexicographical order.
\end{theorem}

Although \Cref{thm:local} is not a perfect analogue of \Cref{thm:local_cpx_jacobi}, our approach for both the general $\theta > 0$ and $\theta = 1$ cases can be viewed as stemming from the same general moment formulas (see \Cref{rem:extension}). The difference between the two cases comes from additional structure available in the $\theta = 1$ case.

The remainder of this section is organized as follows. In \Cref{ssec:ginibre}, we obtain formulas for the Ginibre case via \Cref{thm:obs} under the appropriate limit. We then gather several lemmas in \Cref{ssec:prelim_asymp}, followed by proofs of \Cref{thm:local_cpx,thm:local_cpx_jacobi,thm:local} in \Cref{ssec:local_proof}.

\subsection{Formulas for the Ginibre Case} \label{ssec:ginibre}

We obtain formulas for the joint moments of the squared singular values of complex Ginibre products by taking the appropriate limit of \Cref{thm:obs_cpx}.

\begin{lemma}
Suppose $(\bu^{(T)} := \bu^{(T)}(M))_{T \in \Z_+} \sim \P_1^{(1,M)^{T>0},N}$. Then $(M^T\bu^{(T)})_{T \in \Z_+}$ converges as $M\to\infty$ in finite dimensional distributions to the squared singular values $(\by^{(T)})_{T \in \Z_+}$ of $X_T \cdots X_1$ where $X_1,X_2,\ldots$ are independent, $N\times N$ complex Ginibre matrices.
\end{lemma}

The lemma follows from the observation that $M \to \infty$ corresponds to sending the size of the ambient unitary matrix to infinity. Recall that the matrices $X_T$ in the square case can be obtained as submatrices of a Haar unitary matrix of size $M+N$. After renormalizing by $1/\sqrt{M}$ (renormalizing the squared singular value by $1/M$) the entries of the unitary matrix behave like independent standard complex Gaussians in the limit. We refer the reader to \cite{PR04} for further details on this limit transition.

\begin{proposition} \label{prop:obs_ginibre_cpx}
Suppose $(\by^{(T)})_{T \in \Z_+}$ are the squared singular values of $X_T \cdots X_1$ where $X_1,X_2,\ldots$ are independent, $N\times N$ complex Ginibre matrices. If $T_1 \ge \cdots \ge T_m > 0$ are integers and $c_1,\ldots,c_m \in \R_+$, then
\begin{align*}
\begin{multlined}
\E\left[ \prod_{i=1}^m \fP_{c_i}(\by^{(T_i)}) \right] = \frac{\prod_{i=1}^m (-c_i)^{-1}}{(2\pi\bi)^m} \oint \cdots \oint \prod_{1 \le i < j \le m} \frac{(u_j - u_i)(u_j + c_j - u_i - c_i)}{(u_j - u_i - c_i)(u_j + c_j - u_i)} \\
\times \prod_{i=1}^m \left( \prod_{\ell=1}^N \frac{u_i + \ell - 1}{u_i + c_i + \ell - 1} \cdot \left( \frac{\Gamma(1 - u_i)}{\Gamma(1 - u_i - c_i)} \right)^{T_i} \right) du_i
\end{multlined}
\end{align*}
where the $u_i$-contour $\fU_i$ is positively oriented around $\{-c_i- \ell+1\}_{\ell=1}^N$ but does not enclose $1,2,3,\ldots$, and is enclosed by $\fU_j - c_i, \fU_j + c_j$ for $j > i$; given that such contours exist.
\end{proposition}

\begin{proof}
Suppose $(\bu^{(T)})_{T \in \Z_+} \sim \P_1^{(1,M)^{T>0},N}$. By \Cref{thm:obs_cpx}, we have
\begin{align*}
\E \left[ \prod_{i=1}^m \fP_{c_i}(\bu^{(T_i)}) \right] = \frac{\prod_{i=1}^m (-c_i)^{-1}}{(2\pi\bi)^m} \oint \cdots \oint \prod_{1 \le i < j \le m} \frac{(u_j - u_i)(u_j + c_j - u_i - c_i)}{(u_j - u_i - c_i)(u_j + c_j - u_i)} \\
\times \prod_{i=1}^m \left( \prod_{\ell=1}^N \frac{u_i + \ell - 1}{u_i + c_i + \ell - 1} \cdot \left( \prod_{\ell=1}^M \frac{\ell - u_i - c_i}{\ell - u_i} \right)^{T_i} \right) du_i
\end{align*}
where the $u_i$-contour $\fU_i$ is positively oriented around $\{-c_i -\ell+1\}_{i=1}^N$ but does not enclose $1,\ldots,M$, and is enclosed by $\fU_j - c_i, \fU_j + c_j$ for $j > i$. Taking the Ginibre limit gives
\[ \lim_{M \to\infty} \E \left[ \prod_{i=1}^m M^{c_i T_i} \fP_{c_i}(\bu^{(T_i)}) \right] = \E \left[ \prod_{i=1}^m \fP_{c_i}(\by^{(T_i)}) \right]. \]
This convergence, along with the following asymptotics, then implies this proposition:
\begin{align*}
\lim_{M\to\infty} M^c \prod_{\ell=1}^M \frac{\ell - u - c}{\ell - u} &= \frac{\Gamma(1 - u)}{\Gamma(1 - u - c)} \lim_{M\to\infty} M^c \frac{\Gamma(M + 1 - u - c)}{\Gamma(M + 1 - u)} = \frac{\Gamma(1 - u)}{\Gamma(1 - u - c)}
\end{align*}
uniformly over compact subsets of $\C$ by Stirling's approximation for the Gamma function (or see \Cref{lem:gamma_estimate}).
\end{proof}

\subsection{Preliminary Asymptotics} \label{ssec:prelim_asymp}

\begin{lemma} \label{lem:gamma_estimate}
For any $c > 0$, we have
\[ \log \frac{\Gamma(z+c)}{\Gamma(z)} = c \log(z+c) - \frac{1}{2} \frac{c(c+1)}{z} + R(z) \]
where the remainder satisfies
\[ |R(z)| \le \frac{C_\delta}{|z|^2} \]
over $\{z \in \C: \Re z \ge 0, |z| \ge c(1+\delta)\}$ for any fixed $\delta > 0$ with $C_\delta$ some constant depending on $\delta$.
\end{lemma}
\begin{proof}
By Stirling's approximation, we have
\[ \log \Gamma(z) = \left(z - \frac{1}{2}\right)\log z - z + \frac{1}{2} \log(2\pi) + \frac{1}{12 z} + \wt{R}(z), \quad \quad |\wt{R}(z)| \le \frac{C}{|z|^2} \]
for some uniform constant $C$ over $\{z \in \C: \Re z \ge 0\}$, see \cite{OLBC10}*{p141}. We have
\[ \log \frac{\Gamma(z+c)}{\Gamma(z)} = \left(z-\frac{1}{2}\right) \log\left(1 + \frac{c}{z}\right) - c + c \log(z+c) + R_1(z) \]
where
\[ |R_1(z)| \le \frac{C}{|z|^2} \]
over $\{\Re z \ge 0\}$ for some constant $C$. By expanding the first logarithm, we get
\[ \log \frac{\Gamma(z+c)}{\Gamma(z)} = \left(z-\frac{1}{2}\right) \left( \frac{c}{z} - \frac{1}{2} \frac{c^2}{z^2} \right) - c + c \log(z+c) + R_2(z) = -\frac{1}{2} \frac{c(c+1)}{z} + c \log(z+c) + R_3(z) \]
where for any given $\delta > 0$, we have
\[ \max(|R_2(z)|,|R_3(z)|) \le \frac{C_\delta}{|z|^2} \]
over $\{\Re z \ge 0, |z| \ge c(1+\delta) \}$ for some constant $C_\delta$ depending on $\delta$.
\end{proof}

Define
\begin{align} \label{eq:GL}
{}_cG_L(u) := \frac{L^{-c} \Gamma(L-u)}{\Gamma(L-u-c)}.
\end{align}

\begin{lemma} \label{lem:GL}
Fix $\delta,c,\eta > 0$ and for each $N \in \Z_+$, define the region
\[ \fR_\eta^{(N)} := \{u \in \C: |\Im u| \le \eta,\quad -\eta \le \Re u \le N - \delta \}. \]
Suppose $L:= L(N) \ge N$. Then
\begin{align} \label{eq:GL_conv}
{}_c G_L(u) & = \exp\left[ c\log\left(1 - \frac{u}{L}\right) - \frac{1}{2} \frac{c(c+1)}{L-u-c} + R(L-u-c) \right], \quad \quad u \in \fR_\eta^{(N)}, \\
\label{eq:GL_bd2}
|{}_c G_L(u)| & \le \exp\left( -c \frac{\Re u}{L} + c\frac{\eta}{L}\right), \quad \quad u \in \fR_\eta^{(N)} \cap \{u \in \C: |L-u-c| \ge \delta\}, \quad \mbox{$N$ sufficiently large}
\end{align}
where $R(z)$ is independent of $L$ and satisfies
\[ |R(L-u-c)| \le \frac{C_{\delta,c}}{|L-u-c|^2}, \quad u \in \fR_\eta^{(N)}, \quad L \ge N, \]
If we further assume that $\liminf_{N\to\infty} L/N > 1$, then
\begin{align} \label{eq:GL_bd1}
|{}_c G_L(u)| \ge \exp\left( - \frac{C_{\delta,c}}{L-N} \right), \quad u \in \fR_\eta^{(N)}
\end{align}
for $L - N$ and $N$ sufficiently large.
\end{lemma}

\begin{proof}
By \Cref{lem:gamma_estimate}, we obtain \eqref{eq:GL_conv}. Observe that
\[ |L-u-c| \ge L-\Re u-c \ge L - N - c, \quad u \in \fR_\eta^{(N)} \]
for $L \ge N$. Then \eqref{eq:GL_conv} implies \eqref{eq:GL_bd1}.

Thus we are left to prove \eqref{eq:GL_bd2}. For this, we bound ${}_c G_L(u)$ separately on $\fR_\eta^{(N)} \cap \{\Re u \le L - N_0\}$ and $\fR_\eta^{(N)} \cap \{ L - N_0 \le \Re u \le L\}$ where $N_0$ is a large number independent of $N$, $L$ determined below. We have
\[ \left| \frac{L-\Re u - c}{L - u - c} \right| \le C_{\delta,c,\eta} \]
for $u \in \fR_\eta^{(N)} \cap \{|L-u-c|\ge\delta\}$ and in particular on $u \in \fR_\eta^{(N)} \cap \{\Re u \le L - N_0\}$ where $N_0$ is some fixed number. Thus
\begin{align*}
\Re \left( -\frac{1}{2}  \frac{c(c+1)}{L-u-c} + R(L-u-c) \right) &\le \Re \left( -\frac{1}{2} \frac{c(c+1)}{L-\Re u - c} + \frac{1}{2}\frac{c(c+1)}{L-\Re u - c} - \frac{1}{2}\frac{c(c+1)}{L-u-c} + \frac{C_{\delta,c}}{|L-u-c|^2} \right) \\
& = - \frac{1}{2} \frac{c(c+1)}{L-\Re u - c} + \Re \left( \frac{-c(c+1) \bi \Im u}{2(L - \Re u - c)(L-u-c)} + \frac{C_{\delta,c}}{|L-u-c|^2} \right) \\
& \le - \frac{1}{2} \frac{c(c+1)}{L-\Re u - c} + \frac{C_{\delta,c,\eta}}{|L-\Re u - c|^2}.
\end{align*}
for $u \in \fR_\eta^{(N)} \cap \{\Re u \le L - N_0\}$ where we use the bound on the remainder in the first inequality. By choosing $N_0$ sufficiently large, the first summand on the right hand side dominates. Then \eqref{eq:GL_conv} implies
\[ |{}_c G_L(u)| \le \left|1 - \frac{u}{L} \right|^c \exp\left( \frac{-C_{\delta,N_0}}{L-\Re u - c} \right) \le \left|1 - \frac{u}{L} \right|^c, \quad u \in \fR_\eta^{(N)} \cap \{\Re u \le L - N_0\} \]
where $L \gg N_0$. Since
\[ \log\left| 1 - \frac{u}{L} \right| \le \log \left(1 - \frac{\Re u}{L} + \frac{\eta}{L} \right) \le -\frac{\Re u}{L} + \frac{\eta}{L}, \]
we obtain
\begin{align} \label{eq:GL_bd_bulk}
|{}_c G_L(u)| \le \exp\left( -c\frac{\Re u}{L} + c\frac{\eta}{L}\right), \quad u \in \fR_\eta^{(N)} \cap \{\Re u \le L - N_0\}
\end{align}
for sufficiently large $N$. We also have that \eqref{eq:GL_conv} implies the bound
\[ |{}_c G_L(u)| \le \left( \frac{N_0}{N} e^{C_{\delta,\eta,N_0}}\right)^c < \exp\left( -c\frac{\Re u}{L} \right), \quad u \in \fR_\eta^{(N)} \cap \{|L-u-c| \ge \delta\} \cap \{L - N_0 \le \Re u \le L\}. \]
where the second inequality holds for large $N$ (recall $L \ge N$). Combining the above with \eqref{eq:GL_bd_bulk} gives \eqref{eq:GL_bd2}.
\end{proof}

\subsection{Proofs of Local Theorems} \label{ssec:local_proof}

We prove \Cref{thm:local}, then prove \Cref{thm:local_cpx,thm:local_cpx_jacobi} simultaneously.

\begin{proof}[Proof of \Cref{thm:local}]

By \Cref{thm:obs}, and replacing $u_i$ with $u_i - N \theta$, we get
\[ \E\left[ \prod_{i=1}^m \left( \frac{1}{N} \prod_{\tau = 1}^{T_i} \frac{N+M_\tau+\alpha_\tau-1}{N+\alpha_\tau-1} \right)^{k_i} \fP_{k_i}(\by^{(T_i)}) \right]
= \frac{(-\theta)^{-m}}{(2\pi\bi)^{\sum_i k_i}} \oint \cdots \oint \prod_{i < i'} \cC(U_i,U_{i'}) \prod_{i=1}^m \cB(U_i) I_{T_i}(U_i) \, dU_i. \]
where $U_i := (u_{i,1},\ldots,u_{i,k_i})$ and $I_T := I_T^{(N)}$ is defined by
\[ I_T(U) := \prod_{i=1}^k \frac{1 - \frac{u_i}{\theta N}}{-u_i/\theta} \prod_{\tau=1}^T \frac{1 - \frac{u_i}{\theta(N+\alpha_\tau-1)}}{1 - \frac{u_i}{\theta(N + M_\tau + \alpha_\tau - 1)}} \]
for $U = (u_1,\ldots,u_k)$. The contours are given as follows. The $u_{i,j}$-contour $\fU_{i,j}$ is positively oriented around $0$, avoids the poles at $\theta (N + M_\tau +\alpha_\tau-1)$ for $1 \le \tau \le T_i$ and is enclosed by $\fU_{i',j'} - \theta, \fU_{i',j'} + 1$ whenever $(i',j') > (i,j)$ in lexicographical order. Since $\theta(N + M_\tau +\alpha_\tau-1)$ tends to $\infty$, we may take the contours $\fU_{i,j}$ to be fixed for sufficiently large $N$. The theorem follows by observing
\[ \lim_{N\to\infty} I_{T_i}(U) = \prod_{i=1}^k \exp\left(-\gamma_i u_i/\theta\right) \frac{1}{-u_i/\theta} \]
uniformly for $U = (u_1,\ldots,u_k)$ in compact subsets of $\C^k \setminus \{\vec{0}\}$.
\end{proof}

\begin{proof}[Proofs of \Cref{thm:local_cpx,thm:local_cpx_jacobi}]
Both \Cref{thm:local_cpx,thm:local_cpx_jacobi} are in the setting $\theta = 1$, the difference being that the former is under the assumption that $\by^{(T)}$ are obtained from Ginibre matrices (we refer to this as the \emph{Ginibre case}) and the latter uses Assumption \ref{assum:local} (which we refer to as the \emph{Jacobi case}). We begin by writing our expressions in a common way. Let
\[ E^{(N)} := \left\{ \def\arraystretch{2.8}\begin{array}{ll} \displaystyle  \E\left[ \prod_{i=1}^m \frac{1}{N^{c_i(T_i+1)}} \fP_{c_i}(\by^{(T_i)}) \right] & \mbox{for the Ginibre case,} \\
\displaystyle \E\left[ \prod_{i=1}^m \left( \frac{1}{N} \prod_{\tau=1}^{T_i} \frac{N+M_\tau+\alpha_\tau-1}{N+\alpha_\tau-1} \right)^{c_i} \fP_{c_i}(\by^{(T_i)}) \right] & \mbox{for the Jacobi case.} \end{array} \right. \]
We are interested in the limit of $E^{(N)}$ as $N\to\infty$. By \Cref{thm:obs_cpx}, \Cref{prop:obs_ginibre_cpx}, and replacing $u_i$ with $u_i - N + 1$, we get
\begin{align} \label{eq:ginibre_proof1}
E^{(N)} = \frac{(-1)^m}{(2\pi\bi)^m} \oint \cdots \oint \prod_{1 \le i < j \le m} \frac{(u_j - u_i)(u_j + c_j - u_i - c_i)}{(u_j - u_i - c_i)(u_j + c_j - u_i)} \prod_{i=1}^m I_i(u_i) \, \frac{du_i}{c_i}
\end{align}
where
\[ I_i(u) := I_i^{(N)}(u) :=
\displaystyle \prod_{\ell=1}^N \frac{u + \ell - N}{u + c_i + \ell - N} \cdot \prod_{\tau = 1}^T \frac{{}_{c_i}G_{N+\alpha_\tau-1}(u)}{{}_{c_i}G_{N+M_\tau+\alpha_\tau-1}(u)} . \]
We recall ${}_cG_L(u)$ is defined by \eqref{eq:GL} and for the Ginibre case we take $M_\tau = \infty$ and $\alpha_\tau = 1$ for all $\tau \ge 1$ with ${}_cG_\infty \equiv 1$. The contours in \eqref{eq:ginibre_proof1} are given as follows. The $u_i$-contour $\fU_i^{(N)}$ is positively oriented around $\{-c_i,\ldots,-c_i+N-1\}$ but does not enclose $N+\alpha_\tau-1,\ldots,N+M_\tau + \alpha_\tau-2$ for $1 \le \tau < \infty$, and is enclosed by $\fU_j^{(N)} - c_i, \fU_j^{(N)} + c_j$ for $j > i$; given that such contours exist. Let
\[ \Z_{\vec{c}} := \Z \cup \bigcup_{i=1}^m (\Z - c_i). \]
The assumption that $a := \liminf_{N\to\infty}( \inf_T \alpha_T) > 0$ and $c_i < a/m$ implies we can choose
\begin{align*}
r_1,\ldots,r_m \in (N - 1,N + a - 1) \setminus \Z_{\vec{c}} \quad & \mbox{such that} \quad r_i - r_{i-1} > a/m. \\
s_1,\ldots,s_m \in (-\infty,-a/m) \setminus \Z_{\vec{c}} \quad & \mbox{such that} \quad s_{i-1} - s_i > a/m
\end{align*}
Then $\fU_i^{(N)}$ can be chosen to be a counterclockwise contour around the boundary of the rectangle
\begin{align*}
\left\{u \in \C: |\Im u | \le |s_i|, s_i \le \Re u \le r_i \right\}.
\end{align*}
For $1 \le i \le m$, let $\fU_i$ be the counterclockwise contour around the boundary of the semi-infinite region
\[ \left\{ u \in \C: |\Im u| \le |s_i|, s_i \le \Re u \right\}. \]

The idea of the proof is to use dominated convergence where most of the work in finding a dominating function for $I_i$ is done by \Cref{lem:GL}. We divide the proof into two steps. First, we examine the $N\to\infty$ asymptotics of the function $I_i(u)$. Then, we apply the asymptotics to \eqref{eq:ginibre_proof1} to prove the theorem.

\vspace{2mm}

\textbf{Step 1.} Let
\[ \fR_\eta^{(N)} := \{u \in \C: |\Im u| \le \eta,\quad -\eta \le \Re u \le N - \delta \}. \]
The goal is to control the term $I_i(u)$ as $N\to\infty$ on the region
\[ \fR_\eta^{(N)} \cap \{\mathrm{dist}(u,\Z \cup(\Z \cup c_i)) \ge \delta\} \]
where $\eta > 0$ is some fixed parameter. Throughout this step, all constants depend on $\delta$, $c_i$, and $\eta$. We prove there exist constants $C, C' > 0$ such that
\begin{align} \label{eq:IcT_bd}
|I_i(u)| \le C \exp\left(-C' \Re u \right), \quad u \in \fR_\eta^{(N)} \cap \{\mathrm{dist}(u,\Z \cup (\Z-c_i)) \ge \delta\}
\end{align}
for sufficiently large $N$, and for fixed $u \in \bigcup_{N=1}^\infty \fR_\eta^{(N)} \setminus (\Z \cup(\Z-c_i))$, we have
\begin{align} \label{eq:IcT_conv}
\lim_{N\to\infty} I_i(u) = \frac{\Gamma(-u-c_i)}{\Gamma(-u)} \exp\left[ - \gamma_i \left( \frac{c_i(c_i+1)}{2} + c_i u \right) \right] =: g_i(u)
\end{align}
where we set $\gamma_i = \wh{T}_i$ in the Ginibre case.

We write
\[ I_i(u) = \frac{\Gamma(-u-c_i)}{\Gamma(-u)} \cdot J(u) \]
where
\begin{align} \label{eq:J_term}
J(u) &:= G_N(u) \prod_{\tau=1}^{T_i} \frac{G_{N+\alpha_\tau-1}(u)}{G_{N+M_\tau+\alpha_\tau-1}(u)}.
\end{align}
Then
\begin{align} \label{eq:preJ_bd}
\begin{multlined}
\left| \frac{\Gamma(-u-c_i)}{\Gamma(-u)} \right| = \left| \frac{\sin \pi u}{\sin \pi(u+c_i)} \frac{\Gamma(1+u)}{\Gamma(1+u+c_i)} \right| < C \left|\frac{\Gamma(1+u)}{\Gamma(1+u+c_i)} \right| < C' |1 + u + c_i|^{-c_i}, \\
\mbox{for}~ u \in \fR_\eta^{(N)} \cap \{\mathrm{dist}(u,\Z \cup(\Z - c_i)) \ge \delta\}
\end{multlined}
\end{align}
where the equality follows from Euler's reflection formula, the first inequality from the periodicity of sine in the real direction, and the second inequality from \Cref{lem:gamma_estimate}. Also, \eqref{eq:GL_bd1} and \eqref{eq:GL_bd2} applied to \eqref{eq:J_term} imply
\begin{align} \label{eq:J_bd}
|J(u)| \le C \exp\left( -C' \Re u \right), \quad u \in \fR_\eta^{(N)} \cap \{ \mathrm{dist}(u,\Z \cup(\Z \cup c_i)) \ge \delta\}
\end{align}
for sufficiently large $N$. Combining \eqref{eq:preJ_bd} and \eqref{eq:J_bd} implies \eqref{eq:IcT_bd}. Combining \eqref{eq:J_term} and \eqref{eq:GL_conv} implies \eqref{eq:IcT_conv}.

\vspace{2mm}

\textbf{Step 2.} Let 
\[ \fU_{i,0}^{(N)} := \fU_i^{(N)} \cap \fU_i, \quad \quad \fU_{i,1}^{(N)} := \fU_i^{(N)} \setminus \fU_{i,0}^{(N)} \]
We may rewrite \eqref{eq:ginibre_proof1} as
\begin{align} \label{eq:ginibre_proof2}
E^{(N)} = \sum_{\vec{s} \in \{0,1\}^m} \cI_{\vec{s}}
\end{align}
where
\[ \cI_{\vec{s}} := \frac{(-1)^m}{(2\pi\bi)^m} \oint \cdots \oint \prod_{1 \le i < j \le m} \frac{(u_j - u_i)(u_j + c_j - u_i - c_i)}{(u_j - u_i - c_i)(u_j + c_j - u_i)} \prod_{i=1}^m I_i(u_i) \, \frac{du_i}{c_i} \]
and the $u_i$-contour is $\fU_{i,s_i}^{(N)}$ for $1 \le i \le m$ if $\vec{s} = (s_1,\ldots,s_m)$. We prove that
\begin{align} \label{eq:cI0}
\cI_{\vec{0}} &\to \frac{(-1)^m}{(2\pi\bi)^m} \oint \cdots \oint \prod_{1 \le i < j \le m} \frac{(u_j - u_i)(u_j + c_j - u_i - c_i)}{(u_j - u_i - c_i)(u_j + c_j - u_i)} \prod_{i=1}^m g_i(u_i) \, \frac{du_i}{c_i}, \\ \label{eq:cIs}
\cI_{\vec{s}} &\to 0, \quad\quad \vec{s} \ne \vec{0}
\end{align}
as $N\to\infty$, where the $u_i$-contour in the right hand side of \eqref{eq:cI0} is given by $\fU_i$ for $1 \le i \le m$. By \eqref{eq:ginibre_proof2} and the Cauchy determinant formula, this completes the proof of \Cref{thm:local_cpx,thm:local_cpx_jacobi}. Observe that
\[ \prod_{1 \le i < j \le m} \frac{(u_j - u_i)(u_j + c_j - u_i - c_i)}{(u_j - u_i - c_i)(u_j + c_j - u_i)} \]
is bounded on $(u_1,\ldots,u_m) \in \bigcup_{N=1}^\infty (\fU_1^{(N)} \times \cdots \times \fU_m^{(N)})$. Moreover, for a suitably chosen $\eta$ and $\delta$, we have $\fU_i^{(N)} \subset \fR_\eta^{(N)} \cap \{ \mathrm{dist}(u,\Z\cup(\Z-c_i) \ge \delta\}$. Then \eqref{eq:IcT_bd}, \eqref{eq:IcT_conv}, and dominated convergence imply \eqref{eq:cI0}. To prove \eqref{eq:cIs}, we use the boundedness of the cross-term to write
\[ |\cI_{\vec{s}}| \le C \int_{\fU_{1,s_1}^{(N)}} \! d|u_1| \cdots \int_{\fU_{m,s_m}^{(N)}} \! d|u_m| \cdot \prod_{i=1}^m | I_i(u_i)| = C \prod_{i=1}^m \int_{\fU_{i,s_i}^{(N)}} \! |I_i(u)| d|u|. \]
Using the bound \eqref{eq:IcT_bd}, we have that
\[ \int_{\fU_{i,s_i}^{(N)}} \! |I_i(u)| \, d|u| \]
is of constant order when $s_i = 0$ and is $o(1)$ when $s_i = 1$. Thus \eqref{eq:cIs} follows.
\end{proof}

\appendix \section{} \label{sec:matrix}

We demonstrate that $\beta$-Jacobi product processes with certain parameters can be realized as the squared singular values of products of truncated Haar-distributed $\U_\beta$ matrices for $\beta = 1,2,4$.

Fix $\theta \in \{1/2,1,2\}$, and positive integer parameters $L_1,L_2,\ldots$, $N$, $N_1,N_2,\ldots,$ such that
\[ N \le N_T, \quad \max(N_{T-1},N_T) \le L_T, \quad \quad T \in \Z_+ \]
where $N_0 := N$. Let $U_1,U_2,\ldots$ be independent random matrices such that $U_T$ is a random Haar $\U_{2\theta}(L_T)$ matrix. Let $X_T$ be an $N_T \times N_{T-1}$ submatrix (or \emph{truncation}) of $U_T$,
\[ Y_T := X_T \cdots X_1, \]
and $\vec{y}^{(T)} \in \cR^N$ be the ordered vector of eigenvalues of $Y_T^*Y_T$.

\begin{theorem} \label{thm:matrix_model}
Suppose $\vec{x}^{(1)},\vec{x}^{(2)},\ldots$ are independent random vectors in $\cR^N$ such that $\vec{x}^{(T)} \sim \P_\theta^{\alpha_T,M_T,N}$ where $\alpha_T = N_T - N + 1$ and $M_T = L_T - N_T$. Then $\left(\vec{y}^{(T)}\right)_{T\in \Z^+}$ is equal in distribution to
\[ \left( \vec{x}^{(1)} \boxtimes_{2\theta} \cdots \boxtimes_{2\theta} \vec{x}^{(T)}\right)_{T \in \Z^+}. \]
\end{theorem}

\begin{lemma} \label{lem:switch_product}
Let $n,n_1,n_2,m \in \Z_+$ such that $n \le n_i \le m$, $i = 1,2$. Suppose $X$ and $\wt{X}$ are respectively $n_2 \times n_1$ and $n_2\times n$ truncations of a random Haar $\mathbb{U}_{2\theta}(m)$ matrix. Let $W$ be a fixed $n_1 \times n$ matrix and $\wt{W} := (W^*W)^{1/2}$. If $\sigma(A) \in \cR^N$ denotes the singular values of a matrix $A$, then $\sigma(XW) \overset{d}{=} \sigma(\wt{X} \wt{W})$.
\end{lemma}
\begin{proof}
Let $P_{a\times b}$ denote the $a\times b$ matrix with the $\min(a,b)\times \min(a,b)$ identity matrix in the upper left corner and $0$ elsewhere. The singular value decomposition of $M$ gives
\[ W = U P_{n_1\times n} \Sigma V^* \]
where $\Sigma = \mathrm{diag}(\sigma(W))$, $U \in \U_{2\theta}(n_1),V \in \U_{2\theta}(n)$. Then
\[ (XW)^*(XW) = V \Sigma P_{n\times n_1} U^* X^* X U P_{n_1\times n} \Sigma V^*. \]
Observe that
\[ P_{n\times n_1} U^* X^* X U P_{n_1\times n} \overset{d}{=} \wt{X}^* \wt{X} \overset{d}{=} V^* \wt{X}^* \wt{X} V \]
using the fact that the distributions of $X$ and $\wt{X}$ are invariant under right translation by $\U_{2\theta}$. Thus
\[ (XW)^*(XW) \overset{d}{=} V \Sigma V^* \wt{X}^* \wt{X} V \Sigma V^* = (\wt{X} \wt{W})^* \wt{X} \wt{W} \]
which proves the lemma.
\end{proof}

\begin{proof}[Proof of \Cref{thm:matrix_model}]
Set
\[ \wt{\vec{y}}^{(T)} := \vec{x}^{(1)} \boxtimes_{2\theta} \cdots \boxtimes_{2\theta} \vec{x}^{(T)}. \]
Let $\wt{X}_1,\wt{X}_2,\ldots$ be independent such that $\wt{X}_T$ is an $N_T \times N$ truncation of a Haar $\U_{2\theta}(M_T)$ matrix, and
\[ \wt{Y}_T := (\wt{X}_T^*\wt{X}_T)^{1/2} \cdots (\wt{X}_1^*\wt{X}_1)^{1/2}. \]
By \Cref{prop:U_beta} and the definition of $\boxtimes_{2\theta}$, $\wt{\vec{y}}^{(T)}$ is the squared singular values vector of $\wt{Y}_T$. We have $\vec{y}^{(1)} \overset{d}{=} \wt{\vec{y}}^{(1)}$, and for any fixed $\vec{u} \in \cU^N$,
\[ (\vec{y}^{(T)}|\vec{y}^{(T-1)} = \vec{u}) \overset{d}{=} (\wt{\vec{y}}^{(T)}|\wt{\vec{y}}^{(T-1)} = \vec{u}) \]
by \Cref{lem:switch_product}, thus completing our proof.
\end{proof}

\section{Observables of Schur Processes} \label{sec:obs}

We derive a contour integral formula for joint moments of a special case of \emph{Schur processes} --- Macdonald processes in the case $q = t$.

For $q = t$, the Macdonald symmetric functions become the \emph{Schur functions}
\[ s_\lambda(X) := P_\lambda(X;t,t) = Q_\lambda(X;t,t), \quad s_{\lambda/\mu}(X) := P_{\lambda/\mu}(X;t,t) = Q_{\lambda/\mu}(X;t,t) \]
which are independent of $t$. Thus properties for Macdonald symmetric functions are inherited by the Schur functions. For example, for any countable set of variables $X,Y$, \eqref{eq:branch} implies
\begin{align} \label{eq:schur_branch}
s_{\lambda/\nu}(X,Y) = \sum_{\mu \in \Y} s_{\lambda/\mu}(X) s_{\mu/\nu}(Y).
\end{align}
The Schur functions have an explicit form
\begin{align} \label{eq:alternant}
s_\lambda(x_1,\ldots,x_n) = \frac{\det \begin{pmatrix} x_i^{\lambda_j + n - j} \end{pmatrix}_{i,j=1}^n}{\prod_{1 \le i < j \le n} (x_i - x_j) },
\end{align}
for details see \cite{Mac}*{Chapter I, Sections 3 \& 4}

\begin{definition} \label{def:schur_process}
Suppose $\vec{a} := (a_1,\ldots,a_N) \in \R_+^N$, $\vec{b} \in (b_1,\ldots,b_M) \in \R_+^M$ such that $a_ib_j < 1$ for $1 \le i \le N$, $1 \le j \le M$. Let $\mathbb{SP}_{\vec{a},\vec{b}}$ denote the measure on $\Y^M = \Y \times \cdots \times \Y$ where
\begin{align} \label{eq:schur_process}
\mathbb{SP}_{\vec{a},\vec{b}}(\vec{\lambda}) := \prod_{\substack{1 \le i \le N \\ 1 \le j \le M}} (1 - a_i b_j) \cdot s_{\lambda^M}(a_1,\ldots,a_N) s_{\lambda^M/\lambda^{M-1}}(b_M) \cdots s_{\lambda^2/\lambda^1}(b_2) s_{\lambda^1}(b_1)
\end{align}
for $\vec{\lambda} := (\lambda^1,\ldots,\lambda^M) \in \Y^M$.
\end{definition}

\begin{remark}
The measure $\mathbb{SP}_{\vec{a},\vec{b}}$ is the $q = t$ case of the so-called \emph{ascending Macdonald process}. We have
\[ \sum_{\lambda^1,\ldots,\lambda^M \in \Y} s_{\lambda^M}(a_1,\ldots,a_N) s_{\lambda^M/\lambda^{M-1}}(b_M) \cdots s_{\lambda^2/\lambda^1}(b_2) s_{\lambda^1}(b_1) = \prod_{\substack{1 \le i \le N \\ 1 \le j \le M}} \frac{1}{1 - a_i b_j} \]
as a consequence of \eqref{eq:schur_branch} and the $q = t$ case of \eqref{eq:cauchy1}.
\end{remark}

We prove the following contour integral formula for joint expectations of $\boldsymbol{p}_t$ (recall \Cref{def:p}).

\begin{theorem} \label{thm:schur_obs}
Suppose $\vec{a} := (a_1,\ldots,a_N) \in \R_+^N$, $\vec{b} \in (b_1,\ldots,b_M) \in \R_+^M$ such that $a_ib_j < 1$ for $1 \le i \le M$, $1 \le j \le N$. For real $t_1,\ldots,t_m > 0$, integers $1 \le n_m \le \cdots \le n_1 \le M$, and $\vec{\lambda} \sim \mathbb{SP}_{\vec{a},\vec{b}}$, we have
\begin{align*}
\E\left[\prod_{i=1}^m \boldsymbol{p}_{t_i}(\lambda^{n_i})\right] = \frac{1}{(2\pi\bi)^m} \oint \cdots \oint \prod_{1 \le i < j \le m} \frac{(z_j - z_i)(t_i z_j - t_j z_i)}{(t_i z_j - z_i)(z_j - t_j z_i)} \prod_{i=1}^m \left( \prod_{\ell=1}^N \frac{z_i - a_\ell}{z_i - t_i a_\ell} \cdot \prod_{\ell=1}^{n_i} \frac{1 - t_i^{-1} b_\ell z_i}{1 - b_\ell z_i} \right)  \frac{dz_i}{z_i} 
\end{align*}
where the $z_i$-contour $\fY_i$ is positively oriented around $0, \{t_ia_\ell\}_{\ell=1}^N$ but does not encircle $\{b_\ell^{-1}\}_{\ell=n_i}^M$, and is encircled by $t_j^{-1} \fY_j$, $t_i \fY_j$ for $j > i$; given that such contours exist.
\end{theorem}

\begin{proof}
Let $D_t^{x_1,\ldots,x_n}$ act on functions in $(x_1,\ldots,x_n)$ by
\[ D_t^{x_1,\ldots,x_n} := (1 - t^{-1})\sum_{i=1}^n \prod_{j \ne i} \frac{x_i - t^{-1}x_j}{x_i - x_j} T_{t,x_i} + t^{-n} \]
where $T_{t,x_i}$ is the $t$-shift operator
\[ T_{t,x_i}: f(x_1,\ldots,x_n) \mapsto f(x_1,\ldots,x_{i-1},tx_i,x_{i+1},\ldots,x_n). \]
Using \eqref{eq:alternant}, we have
\[ D_t^{x_1,\ldots,x_n} s_\lambda(x_1,\ldots,x_n) = \boldsymbol{p}_t(\lambda) s_\lambda(x_1,\ldots,x_n), \quad \quad \lambda \in \Y_n. \]
Then \eqref{eq:schur_branch} and \eqref{eq:schur_process} imply
\[ \prod_{\substack{1 \le i \le N \\ 1 \le j \le M}} (1 - a_i b_j) \cdot D_{t_m}^{b_1,\ldots,b_{n_m}} \cdots D_{t_1}^{b_1,\ldots,b_{n_1}} \prod_{\substack{1 \le i \le N \\ 1 \le j \le M}} \frac{1}{1 - a_i b_j} = \E \left[ \prod_{i=1}^m \boldsymbol{p}_{t_i}(\lambda^{n_i}) \right]. \]
On the other hand, if $n \le M$ and $f$ is analytic in a neighborhood of $\{0\} \cup \{b_i\}_{i=1}^n$ such that $f(tw)/f(w)$ is also analytic in this neighborhood, then by residue expansion,
\[ D_t^{b_1,\ldots,b_n} \prod_{i=1}^M f(b_i) = \prod_{i=1}^M f(b_i) \cdot \frac{1}{2\pi\bi} \oint \left( \prod_{j=1}^n \frac{w - t^{-1} b_j}{w - b_j} \right) \frac{f(tw)}{f(w)}  \frac{dw}{w} \]
where the contour is positively oriented around $0$, $\{b_i\}_{i=1}^n$, but no other poles of the integrand. By iterating, we obtain
\begin{align*}
\E \left[ \prod_{i=1}^m \boldsymbol{p}_{t_i}(\lambda^{n_i}) \right] & = \prod_{\substack{1 \le i \le N \\ 1 \le j \le M}} (1 - a_i b_j) \cdot D_{t_m}^{b_1,\ldots,b_{n_m}} \cdots D_{t_1}^{b_1,\ldots,b_{n_1}} \prod_{\substack{1 \le i \le N \\ 1 \le j \le M}} \frac{1}{1 - a_i b_j} \\
& = \frac{1}{(2\pi\bi)^m} \oint \cdots \oint \prod_{1 \le i < j \le m} \frac{(w_i - w_j)(w_i - \frac{t_j}{t_i}w_j)}{(w_i - \frac{1}{t_i} w_j)(w_i - t_j w_j)}  \prod_{i=1}^m \left( \prod_{j=1}^{n_i} \frac{w_i - \frac{1}{t_i} b_j}{w_i - b_j} \prod_{j=1}^N \frac{1 - a_j w_i}{1 - t_i a_j w_i} \right) \frac{dw_i}{w_i}
\end{align*}
where the $w_j$-contour is positively oriented around $0$, $\{b_j\}_{j=1}^{n_i}$, but no other poles of the integrand, and enclosed by $t_i w_i$, $t_j^{-1} w_i$ for $i < j$; assuming such contours exist. Taking $z_i = w_i^{-1}$ completes the proof.
\end{proof}

\section{} \label{sec:appendix}
We recall the notion of cumulants and some basic properties.

\begin{definition} \label{def:cumulant1}
For any finite set $S$, let $\Theta_S$ be the collection of all set partitions of $S$, that is
\[ \Theta_S := \left\{ \{S_1,\ldots,S_d\}: d > 0, ~\bigcup_{i=1}^d S_i = S, ~S_i \cap S_j = \emptyset~\forall i \ne j,~S_i \ne \emptyset ~\forall i \in [[1,d]] \right\}. \]
For a random vector $\vec{u} = (u_1,\ldots,u_m)$ and any $v_1,\ldots,v_\nu \in \{u_1,\ldots,u_m\}$, define the \emph{(order $\nu$) cumulant}
\begin{align} \label{eq:cumulant1}
\kappa(v_1,\ldots,v_\nu) := \sum_{\substack{d > 0 \\ \{S_1,\ldots,S_d\} \in \Theta_{[[1,\nu]]}}} (-1)^{d-1} (d-1)! \prod_{\ell=1}^d \E \left[ \prod_{i \in S_\ell} v_i \right].
\end{align}
\end{definition}

The definition implies that for any random vector $\vec{u}$, the existence of all cumulants of order up to $\nu$ is equivalent to the existence of all moments of order up to $\nu$. Note that the cumulants of order $2$ are exactly the covariances:
\[ \kappa(v_1,v_2) = \cov(v_1,v_2). \]
We can also express the cumulant as
\begin{align} \label{eq:cumulant2}
\kappa(v_1,\ldots,v_\nu) = (-\bi)^\nu \left. \frac{\partial^\nu}{\partial t_1 \cdots \partial t_m} \log \E \left[ \exp\left( \bi \sum_{j=1}^\nu t_j v_j \right) \right] \right|_{t_1 = \cdots = t_\nu = 0}.
\end{align}
For further details see \cite{PT}*{Sections 3.1 \& 3.2} wherein the agreement between \eqref{eq:cumulant1} and \eqref{eq:cumulant2} is shown by taking \eqref{eq:cumulant2} as the definition and proving \eqref{eq:cumulant1}. Note that \eqref{eq:cumulant2} implies

\begin{lemma} \label{lem:gauss_cumulant}
A random vector is Gaussian if and only if all cumulants of order $\ge 3$ vanish.
\end{lemma}

We have a useful inversion lemma

\begin{lemma} \label{lem:cumulant_inverse}
If $K$ and $E$ are complex-valued functions on the set of nonempty subsets of $[[1,\nu]]$ such that
\[ E(S) = \sum_{\substack{d > 0 \\ \{S_1,\ldots,S_d\} \in \Theta_S}} \prod_{i=1}^d K(S_i) \quad \quad \mbox{for all nonempty $S \subset [[1,\nu]]$,} \]
then
\begin{align} \label{eq:KS}
K(S) = \sum_{\substack{d > 0 \\ \{S_1,\ldots,S_d\} \in \Theta_S}} (-1)^{d-1} (d-1)! \prod_{i=1}^d E(S_i).
\end{align}
\end{lemma}

\begin{proof}
Let $E_{n_1,\ldots,n_\nu} \in \C$ with $E_{0,\ldots,0} = 1$. Define the formal power series
\begin{align*}
E(t_1,\ldots,t_\nu) &:= \sum_{n_1,\ldots,n_\nu \ge 0} \frac{E_{n_1,\ldots,n_\nu}}{n_1! \cdots n_\nu!} t_1^{n_1} \cdots t_\nu^{n_\nu} \\
K(t_1,\ldots,t_\nu) &:= \sum_{n_1,\ldots,n_\nu \ge 0} \frac{K_{n_1,\ldots,n_\nu}}{n_1! \cdots n_\nu!} t_1^{n_1} \cdots t_\nu^{n_\nu} := \log E(t_1,\ldots,t_\nu).
\end{align*}
For each nonempty $S \subset [[1,\nu]]$, let $K_{n_1,\ldots,n_\nu} = K(S)$ if $n_j = \1_{j\in S}$. The relation $E(\vec{t}) = e^{K(\vec{t})}$ implies
\[ E_{n_1,\ldots,n_\nu} = \sum_{\substack{d > 0 \\ \{S_1,\ldots,S_d\} \in \Theta_S}} \prod_{\ell=1}^d K(S_\ell) = E(S) \]
for any nonempty $S \subset [[1,\nu]]$ where $n_j = \1_{j\in S}$. The relation $K(\vec{t}) = \log E(\vec{t})$ now implies \eqref{eq:KS}.
\end{proof}

\bibliography{mybib}

\end{document}